\documentclass[12pt]{amsart}
\usepackage[margin=1.25in]{geometry}
\usepackage{latexsym}
\usepackage{amsmath}
\usepackage{amssymb}
\usepackage{amsthm}
\usepackage{stmaryrd}
\usepackage{amscd}
\usepackage{bm}
\usepackage{enumerate}
\usepackage{enumitem}
\usepackage{amssymb} 
\usepackage{mathrsfs}
\usepackage[all]{xy}
\usepackage{color}
\usepackage{graphicx}
\usepackage{mathtools}
\usepackage{comment}
\usepackage{url}
\usepackage[pagebackref]{hyperref}

\makeatletter
  \renewcommand{\theequation}{\arabic{section}.\arabic{thm}.\arabic{equation}}
  \@addtoreset{equation}{thm}
  \makeatother

\title{BCM-regularity of diagonal hypersurfaces and plus-pure thresholds in mixed characteristic}

\author{Tatsuki Yamaguchi}
\address{Department of Mathematics, School of Science, Institute of Science Tokyo, 2-12-1 Ookayama, Meguro, Tokyo 152-8551}
\email{yamaguchi.t.bp@m.titech.ac.jp}

\subjclass[2020]{}

\def\ge{\geqslant}
\def\le{\leqslant}
\def\phi{\varphi}
\def\epsilon{\varepsilon}
\def\tilde{\widetilde}

\def\mapsto{\longmapsto}

\def\Hom{\operatorname{Hom}}

\def\max{\operatorname{max}}

\def\ppt{\operatorname{ppt}}

\def\m{{\mathfrak m}}
\def\n{{\mathfrak n}}
\def\ba{{\mathfrak a}}

\newcommand{\N}{\mathbb{N}}
\newcommand{\Q}{\mathbb{Q}} 
\newcommand{\C}{\mathbb{C}} 
\newcommand{\R}{\mathbb{R}} 
\newcommand{\Z}{\mathbb{Z}}

\theoremstyle{plain}
\newtheorem{thm}{Theorem}[section] 
\newtheorem{cor}[thm]{Corollary}
\newtheorem{prop}[thm]{Proposition}

\newtheorem{lem}[thm]{Lemma}
\theoremstyle{definition} 
\newtheorem{defn}[thm]{Definition}
\newtheorem{eg}[thm]{Example}
\theoremstyle{remark}
\newtheorem{rem}[thm]{Remark}
\newtheorem{ques}[thm]{Question}
\newtheorem{setting}[thm]{Setting}
\newtheorem{notation}[thm]{Notation} 
\newtheorem*{cl}{Claim}
\newtheorem*{clproof}{Proof of Claim}

\newtheorem*{acknowledgement}{Acknowledgments}

\begin{document}

\tolerance = 9999

\begin{abstract}
We introduce a new method for computing plus-pure thresholds, a mixed-characteristic analogue of both log canonical thresholds and $F$-pure thresholds. We obtain some necessary conditions and some sufficient conditions for BCM-regularity of Fermat-type hypersurfaces. We also establish lower bounds for plus-pure thresholds of diagonal hypersurfaces in mixed characteristic. Furthermore, we give bounds for plus-pure thresholds of hypersurfaces in mixed characteristic $(0,2)$ using splitting-order sequences, introduced by Yoshikawa. As an application, we classify BCM-regular diagonal hypersurfaces in mixed characteristic $(0,2)$.
\end{abstract}

\maketitle
\markboth{TATSUKI YAMAGUCHI}{BCM-REGULARITY OF HYPERSURFACES AND PLUS-PURE THRESHOLDS}

\section{Introduction}
The log canonical threshold is an important invariant in birational geometry, defined in terms of resolutions of singularities. It plays a significant role in theories such as the minimal model program (\cite{KM98}, \cite{LazarsfeldII}) and the theory of normalized volumes (\cite{Liu18}).

In positive characteristic, the Frobenius morphism provides a natural framework for defining useful classes of singularities, such as $F$-pure and strongly $F$-regular singularities (\cite{Fedder83}, \cite{Hochster-Huneke90}, \cite{Hochster-Huneke94}). The $F$-pure threshold, introduced by Takagi and Watanabe \cite{Takagi-Watanabe04}, is an invariant associated to a pair $(R,\ba)$, where $R$ is a Noetherian ring of positive characteristic and $\ba\subseteq R$ is a nonzero ideal, defined as the critical value $t$ such that the pair $(R,\ba^t)$ is $F$-pure. The $F$-pure threshold is regarded as a positive-characteristic analogue of the log canonical threshold.

 On the other hand, in mixed characteristic, perfectoid techniques introduced by Scholze \cite{Scholze} have played a central role in the development of singularity theory. Ma and Schwede \cite{Ma-Schwede21} introduced the notion of BCM-regularity as a mixed-characteristic analogue of strong $F$-regularity, while Bhatt, Ma, Patakfalvi, Schwede, Tucker, Waldron and Witaszek \cite{BMPSTWW24} introduced the notion of perfectoid purity as a mixed-characteristic counterpart of $F$-purity. Algebraic geometry in mixed characteristic has been actively studied using these techniques (see, e.g., \cite{BMPSTWW23}, \cite{Takamatsu-Yoshikawa23}, \cite{HLS24}).

  The plus-pure threshold, coined in \cite{CPQGST25}, is an invariant defined as the first jumping number of a variant of mixed-characteristic test ideals. Specifically, if $(R,\m)$ is a regular local ring of residue characteristic $p>0$ and $f\in \m$ is a nonzero element, then the plus-pure threshold $\operatorname{ppt} (f)$ of $f$ is defined as
\[
	\sup\{t\in \Q_{\ge 0}\mid \text{$R\xrightarrow{f^t} R^+$ is pure}\},
\]
where $R^+$ denotes the absolute integral closure of $R$. The invariant can be viewed as a special case of BCM-thresholds introduced by Rodr\'{i}guez-Villalobos \cite{Rodriguez-Villalobos25}. At present, explicit computations of plus-pure thresholds remain limited (see \cite{CPQGST25} and \cite{BJPRMSS25} for recent progress). The aim of this paper is to provide further examples of such computations.

	For $F$-pure thresholds, Fedder's criterion (\cite{Fedder83}) is the main tool for computation. Recently, Yoshikawa \cite{Yoshikawa25} extended the notion of quasi-$F$-split singularities to the mixed-characteristic setting and showed that such singularities are perfectoid pure. He also introduced splitting-order sequences and investigated perfectoid pure singularities using these sequences (\cite{Yoshikawa25b}). In that work, he provided an algorithm to determine whether a hypersurface in mixed characteristic $(0,2)$ is perfectoid pure. In this paper, we also study the relationship between splitting-order sequences and plus-pure thresholds in mixed characteristic $(0,2)$.

 First, we focus on Fermat-type hypersurfaces as toy examples. Even in this relatively simple setting, characterizing BCM-regularity is challenging. We give necessary conditions as well as sufficient conditions, which is enough for determining the BCM-regularity in the cases $p=2,3$.
\begin{thm}[Theorem \ref{thm Fermat type}]
	Let $p$ be a prime number, $n\ge2$, $d$ be positive integers and $R:=\Z_p\llbracket x_0,\dots, x_n\rrbracket/(x_0^d+\dots+x_n^d)$.
\begin{enumerate}
	\item $R$ is BCM-regular if one of the following conditions holds.
		\begin{enumerate}
			\item $d\le \min\{n,p\}$.
			\item $d\le n$ and there exists a positive integer $a$ such that $1\le a \le p-1$ and one of the following holds:
			\begin{enumerate}
				\item $n=\lfloor (p^2-1)/a\rfloor.$
				\item $an=p^2-p$.
				\item $an=p^2-p-1$.
				\item $an=p^2-a-1$.
			\end{enumerate}
		\end{enumerate}
	\item $R$ is not $+$-regular if one of the following conditions holds.
		\begin{enumerate}
			\item $d>n$.
			\item $d\ge p^2$.
			\item There exists a positive integer such that $1\le a \le p-1$ and
			\[
				\lceil \frac{p^2}{a+1} \rceil \le d \le n \le \lfloor \frac{p^2-p-1}{a}\rfloor-1. 
			\]
		\end{enumerate}
\end{enumerate}
\end{thm}
\begin{eg}[Examples \ref{eg p=3} and \ref{eg p=5}]
Let the notation be as in the above theorem.
\begin{enumerate}
\item Suppose that $p=2,3$. $R$ is BCM-regular if and only if $d\le \min\{n, p^2-1\}$.
\item Suppose that $p=5$ and $(d,n)\neq (21,21), (21,22),(22,22)$. Then $R$ is BCM-regular if and only if $d$ and $n$ satisfy one of the following conditions:
\begin{enumerate}
	\item $d\le \min\{12,n\}$.
	\item $d\le \min\{24,n\}$ and $19\le n$.
\end{enumerate}
\end{enumerate}
\end{eg}

Next, we turn to the computation of plus-pure thresholds. We establish a lower bound for the plus-pure thresholds of diagonal hypersurfaces in mixed characteristic, analogous to Hern\'{a}ndez's result on $F$-pure thresholds of diagonal hypersurfaces in positive characteristic (\cite{Hernandez15}).

\begin{thm}[{Theorem \ref{diagonal hypersurface brt}}]
	Let $p$ be a prime number, $n, d_1,\dots,d_n\ge 2$ be positive integers, $S:=\Z_p\llbracket x_1,\dots,x_n \rrbracket$ and $f:=x_1^{d_1}+\dots+x_n^{d_n}$. Suppose that there exist non-negative rational numbers $\alpha, \alpha_1,\dots, \alpha_n$ such that $\alpha=\alpha_1+\dots+\alpha_n\le 1$ and $d_i\alpha_i \le 1$ for $i=1,\dots,n$. Moreover, assume that there is at most one carry at  each digit when adding the $p$-adic expansions of $\alpha_1, \dots, \alpha_n$ as in Notation \ref{notation p-adic expansion}. Then $\operatorname{ppt}(f)\ge \alpha$.
\end{thm}
As an application, we obtain the following corollary, which gives an affirmative answer to \cite[Question 4.4]{CPQGST25}.
\begin{cor}[{Corollary \ref{binomial ppt}}]
Let $p$ be a prime number, $a, b\ge 2$ be positive integers, $S:=\Z_p\llbracket x,y \rrbracket$ and $f:=x^a+y^b$.
Then
\[
	\operatorname{ppt}(f)=\frac{1}{a}+\frac{1}{b}.
\]
\end{cor}

In mixed characteristic $(0,2)$, the computation of plus-pure thresholds is more tractable, and we obtain bounds on plus-pure thresholds using splitting-order sequences (Theorem \ref{theorem splitting-order sequences and ppt}).

As an application of these theorems, we characterize which diagonal hypersurfaces are BCM-regular or perfectoid pure in mixed characteristic $(0,2)$ (Theorem \ref{theorem classification} and Corollary \ref{cor classification}).

In the above results, although the base ring is assumed to be $\Z_p$, the arguments extend straightforwardly to the case where the base ring is the ring of $p$-typical Witt vectors $W(k)$ of a perfect field $k$ of characteristic $p>0$.

The paper is organized as follows. In Section 2, we review the basic notions of singularity theory in mixed characteristic and their fundamental properties. In Section 3, we develop new techniques for computations in mixed characteristic. In Section 4, we study the BCM-regularity of Fermat-type hypersurfaces. In Section 5, we compute the plus-pure thresholds of hypersurfaces and classify BCM-regular and perfectoid pure diagonal hypersurfaces in mixed characteristic $(0,2)$.

\begin{acknowledgement}
 The author would like to express his gratitude to Shunsuke Takagi and Shou Yoshikawa for valuable discussion. The author is also grateful to Eamon Quinlan-Gallego and Karl Schwede for their comments on the manuscript.
 The author was partially supported by JSPS KAKENHI Grant Number 24KJ1040.
\end{acknowledgement}

\section{Preliminaries}
In this section, we review basic notions of singularities in mixed characteristic and prove several elementary statements about plus-pure thresholds. For the definition of perfectoid rings, we refer the reader to \cite[Section 3]{BMS18}.
\begin{defn}
	Let $(R,\m)$ be a Noetherian local domain of dimension $d$. A ring $B$ is said to be a \textit{(balanced) big Cohen--Macaulay $R^+$-algebra} (\textit{BCM $R^+$-algebra} for short) if $B$ is an $R^+$-algebra and any system of parameters $x_1,\dots,x_d$ is a regular sequence on $B$.
\end{defn}
\begin{rem}
    \begin{enumerate}
    \item Suppose that $S$ is a Noetherian local domain such that $S$ is an integral extension of $R$ in $R^+$. If $B$ is a BCM $R^+$-algebra, then $B$ is a BCM $S^+$-algebra (\cite[Corollary 4.6]{Dietz07}).
    \item If, in addition, $R$ is an excellent ring of residue characteristic $p>0$, the $p$-adic completion of $R^+$ is a BCM $R^+$-algebra (see \cite{HH92} in positive characteristic and \cite{Bhatt21} in mixed characteristic).
    \end{enumerate}
\end{rem}
We recall the basic classes of singularities and the invariants needed later. An $R$-linear map $f:M\to N$ is said to be \textit{pure} if for any $R$-module $L$, $f\otimes_RL:M\otimes_{R}L\to N\otimes_RL$ is injective.
\begin{defn}[{\cite[Definition 6.9]{Ma-Schwede21}, \cite[Definition 5.1.3]{CLMST26}}] 
	Let $(R,\m)$ be a Noetherian complete normal $\Q$-Gorenstein local domain of residue characteristic $p>0$, $f\in \m$ be a nonzero element and $t\ge 0$ be a rational number. 
\begin{enumerate}
\item A pair $(R,f^t)$ is said to be \textit{$+$-regular} if $R\xrightarrow{f^t}R^+$ is pure.
\item A pair $(R,f^t)$ is said to be \textit{(perfectoid) BCM-regular} if for any perfectoid BCM $R^+$-algebra, $R\xrightarrow{f^t} B$ is pure.
\end{enumerate}
 We say that $R$ is $+$-regular (resp. BCM-regular) if the pair $(R,1^t)$ is $+$-regular (resp. BCM-regular).
\end{defn}
\begin{rem}
	\begin{enumerate}
		\item Since $R^+$ contains $n$-th roots of $f$ for any $n\in \N_{>0}$ and they differ only up to multiplication by a unit, the above definitions are well defined.
		\item In this paper, we mainly discuss hypersurfaces, so the above definition is sufficient for our purposes. If the ring $R$ is not $\Q$-Gorenstein, there are several versions of $+$-regularity and BCM-regularities (see \cite[Section 5.3]{CLMST26}).
	\end{enumerate}
\end{rem}
\begin{defn}[{\cite[Definition 2.1]{CPQGST25}}] \label{definition thresholds}
	Let $(R,\m)$ be a Noetherian complete normal $\Q$-Gorenstein local domain of residue characteristic $p>0$, for any nonzero element $f\in \m$. 
 The \textit{plus-pure threshold} $\operatorname{ppt}(R,f)$ of $(R,f)$ is defined as
\[
	\sup\{t\in \Q_{\ge 0}\mid \text{$(R,f^t)$ is $+$-regular}\}.
\]
If the ring $R$ is clear from the context, we use $\operatorname{ppt}(f)$ to denote $\operatorname{ppt}(R,f)$ and $\ppt(f)$ to denote $\ppt(R,f)$.
\end{defn}
We also recall the notion of perfectoid purity.
\begin{defn}[{\cite[Definition 4.1]{BMPSTWW24}}]
    Let $R$ be a Noetherian complete local ring of residue characteristic $p>0$. Then $R$ is \textit{perfectoid pure} if there exists a perfectoid $R$-algebra such that $R\to B$ is pure.
\end{defn}
\begin{prop} \label{BCM-regularity and brt}
	Let $m, n$ be positive integers, let $S:=\Z_p\llbracket x_1,\dots,x_n\rrbracket$ and let $f\in (p,x_1,\dots,x_n)S$. Suppose that $R:=\Z_p\llbracket x_0 \rrbracket/(x_0^m+f)$ is a domain. The following three conditions are equivalent:
	\begin{enumerate}
		\item $R$ is BCM-regular.
		\item $f^{\frac{m-1}{m}}\notin (p,x_1,\dots,x_n)B$ for any perfectoid BCM $S^+$-algebra.
		\item $\ppt(S,f)>(m-1)/m$.
	\end{enumerate}
\end{prop}
\begin{proof}
	Since $R$ is a finite extension of $S$, it follows that $R^+\cong S^+$ as an $S$-algebra. We can choose an isomorphism fitting into the following commutative diagram:
\[
	\xymatrix{
		S \ar[r] \ar[d] & R \ar[d]\\
		S^+ \ar[r]^{\cong}_{x_0\mapsto (-f)^{\frac{1}{m}}} & R^+,
	}
\]
where $(-f)^{1/m}$ denotes an $m$-th root of ${-f}$ in $R^+$.
Let $B$ be any integral perfectoid $S^+$-algebra. Then $B$ is also a BCM $R^+$-algebra via the chosen isomorphism $R^+ \cong S^+$. Since $R$ is Gorenstein, $R\to M$ is pure if and only if $H_\m^{d}(R)\to H_\m^{d}(M)$ is injective for any $R$-module $M$, where $d=n+1=\operatorname{dim} R$ and $\m$ is the maximal ideal of $R$. Since $R$ is Gorenstein, the socle of $H_\m^d(R)$ is a one-dimensional $\mathbb{F}_p$-vector space. Moreover, $\left[ \frac{x_0^{m-1}}{px_1\cdots x_n}\right]$ is a generator of the socle. Therefore, $R\to B$ is pure if and only if 
\[
	\left[ \frac{(-f)^{\frac{m-1}{m}}}{px_1\cdots x_n} \right] \neq 0 \in H_\m^d(B),
\]
which is equivalent to the condition that $f^{(m-1)/m}\notin (p,x_1,\dots,x_n)B$. Hence, condition (1) is equivalent to condition (2). 
 
Suppose that condition (2) holds. By \cite[Proposition 6.10]{Ma-Schwede21}, $(S,f^{(m-1)/m+\epsilon})$ is BCM-regular for any rational number $0<\epsilon \ll 1$.
Hence, $\ppt(S,f)>(m-1)/m$. 

Conversely, suppose that condition (3) holds. $(R,f^{(m-1)/m+\epsilon})$ is $+$-regular for any rational number $0<\epsilon \ll 1$.
Since $S$ is regular, it follows from \cite[Theorem 8.11]{BMPSTWW25} that this is equivalent to $(S,f^{(m-1)/m})$ being BCM-regular. Hence, condition (2) holds.

\end{proof}

\begin{rem} \label{remark +-regular and ppt}
Similarly, in the same setting as Proposition \ref{BCM-regularity and brt}, the following two conditions are equivalent:
\begin{enumerate}
	\item $R$ is $+$-regular.
	\item $f^{\frac{m-1}{m}}\notin (p,x_1,\dots,x_n)S^+$.
\end{enumerate}
\end{rem}

The following proposition shows that the BCM-regularity for a certain class of rings implies the log terminality for the corresponding rings in equal characteristic zero.
\begin{lem} \label{+-regular and log terminal}
	Let $n$ be a positive integer and $R:=(\Z_p[x_1,\dots,x_n]/(f))_{(p, x_1,\dots,x_n)}$, where $f\in (x_1,\dots,x_n)\Z[x_1,\dots,x_n]$. If $\widehat{R} \cong \Z_p\llbracket x_1,\dots,x_n\rrbracket/(f)$ is $+$-regular, then $(\C[x_1,\dots,x_n]/(f))_{(x_1,\dots,x_n)}$ has log terminal singularities.
\end{lem}
\begin{proof}
	Suppose that $\widehat{R}$ is $+$-regular. We have a commutative diagram
\[
	\xymatrix{
		R \ar[r] \ar[d] & R^+ \ar[d] \\
		\widehat{R} \ar[r] & (\widehat{R})^{+}.
	}
\]
Since the morphisms $R\to \widehat{R}$ and $\widehat{R} \to (\widehat{R})^+$ are pure, the morphism $R\to R^+$ is also pure.
Hence, $R\to S$ splits for any finite extension $S$ of $R$ contained in $R^+$.
By \cite[Corollary 7.18]{BMPSTWW25}, we have $\mathcal{J}(R[1/p])=R[1/p]$, where $\mathcal{J}(R[1/p])$ is the multiplier ideal of $R[1/p]$. Hence, $R[1/p]$ has log terminal singularities. Since $(\Q_p[x_1,\dots,x_n]/(f))_{(x_1,\dots,x_n)}$ is a localization of $R[1/p]$, the ring $(\Q_p[x_1,\dots,x_n]/(f))_{(x_1,\dots,x_n)}$ has log terminal singularities. Log terminal singularities over a field of characteristic zero are geometrically log terminal by \cite[Proposition 2.15]{Kol13}. Hence, $(\C[x_1,\dots,x_n]/(f))_{(x_1,\dots,x_n)}$ has log terminal singularities.
\end{proof}
\begin{cor} \label{+-regular and canonical diagonal hypersurface}
	Let $n, d_1,\dots,d_n$ be positive integers, $f:=x_1^{d_1}+\dots+x_n^{d_n}$ and $R:=\Z_p\llbracket x_1,\dots,x_n \rrbracket/(f)$. If $R$ is $+$-regular, then 
\[
	\frac{1}{d_1}+\cdots+\frac{1}{d_n}>1.
\]
\end{cor}
\begin{proof}
	By Lemma \ref{+-regular and log terminal}, the ring $(\C[x_1,\dots,x_n]/(x_1^{d_1}+\dots+x_n^{d_n}))_{(x_1,\dots,x_n)}$ has log terminal singularities. Since $(\C[x_1,\dots,x_n]/(x_1^{d_1}+\dots+x_n^{d_n}))_{(x_1,\dots,x_n)}$ is Gorenstein, it has canonical singularities by \cite[Corollary 5.24]{KM98}. Hence, 
\[
	\frac{1}{d_1}+\cdots+\frac{1}{d_n}>1	
\]
by \cite[Proposition (4.3)]{Reid80}.
\end{proof}
 We conclude this section with a proposition for later use.
\begin{prop} \label{proposition index inequality}
	Let $p$ be a prime number, $n\ge 3$ be an integer and $a_i, b_i$ be positive integers such that $a_i\le b_i$ for any $i=1,\dots,n$. If $\Z_p\llbracket x_1,\dots,x_n \rrbracket/(x_1^{b_1}+\dots +x_n^{b_n})$ is BCM-regular (resp. $+$-regular), then $\Z_p\llbracket x_1,\dots,x_n \rrbracket/(x_1^{a_1}+\dots+x_n^{a_n})$ is BCM-regular (resp. $+$-regular).
\end{prop}
\begin{proof}
	First, we give the proof for $+$-regularity.
	Let $S:=\Z_p\llbracket x_2,\dots,x_n \rrbracket$.
	 Consider an injective ring homomorphism
	\[
		\sigma:S \to S[x_2^{\frac{1}{a_2}},\dots, x_n^{\frac{1}{a_n}}]; x_i \mapsto x_i^{\frac{b_i}{a_i}},
	\]
	where $x_i^{1/a_i}$ denotes an $a_i$-th root of $x_i$ in $S^+$. Then there exists an injective ring homomorphism $\sigma':S^+\to S^+$ fitting into the following commutative diagram
\[
	\xymatrix{
		S \ar[d]_{\iota} \ar[r]^-{\sigma} & S[x_2^{\frac{1}{a_2}},\dots, x_n^{\frac{1}{a_n}}]\ar[d]^{\iota'}\\
		S^+ \ar[r]^{\sigma'} & S^+,
	}
\]
	where $\iota$ and $\iota'$ are natural injections.
Suppose that $\Z_p\llbracket x_1,\dots,x_n\rrbracket/(x_1^{a_1}+\dots+x_n^{a_n})$ is not $+$-regular. Since $\Z_p\llbracket x_1,\dots,x_n\rrbracket/(x_1^{a_1}+\dots+x_n^{a_n})$ is a domain, we have
	\[
		(x_2^{a_2}+\dots+x_n^{a_n})^{\frac{a_1-1}{a_1}}\in (p,x_2,\dots,x_n)S^+
	\]
	by Remark \ref{remark +-regular and ppt}. Hence, 
	\[
		\sigma'((x_2^{a_2}+\dots+x_n^{a_n})^{\frac{a_1-1}{a_1}})\in (p,x_2^{\frac{b_2}{a_2}},\dots,x_n^{\frac{b_n}{a_n}})S^+.
	\]
	Then $(\sigma'((x_2^{a_2}+\dots+x_n^{a_n})^{(a_1-1)/a_1}))^{a_1}=(x_2^{b_2}+\dots+x_n^{b_n})^{a_1-1}$. Since $(b_1-1)/b_1\ge(a_1-1)/a_1$, we have
\[
	(x_2^{b_2}+\dots+x_n^{b_n})^{\frac{b_1-1}{b_1}}\in (p,x_2^{\frac{b_2}{a_2}},\dots,x_n^{\frac{b_n}{a_n}})S^+ \subseteq (p,x_2,\dots,x_n)S^+.
\]
By Remark \ref{remark +-regular and ppt} again, $\Z_p\llbracket x_1,\dots,x_n \rrbracket/(x_1^{b_1}+\dots +x_n^{b_n})$ is not $+$-regular. 

Next, suppose that $\Z_p\llbracket x_1,\dots,x_n\rrbracket/(x_1^{a_1}+\dots+x_n^{a_n})$ is not BCM-regular. Then, by Proposition \ref{BCM-regularity and brt}, there exists a perfectoid BCM $S^+$-algebra $B$ such that
\[
	(x_2^{a_2}+\dots+x_n^{a_n})^{\frac{a_1-1}{a_1}}\in (p,x_2,\dots,x_n)B.
\]
By \cite[Theorem A.5]{MSTWW22}, there exists a perfectoid BCM $S^+$-algebra $C$ fitting into the following commutative diagram:
\[
		\xymatrix{
			S \ar[r]^-{\sigma} \ar[d]_{\iota} & S[x_2^{\frac{1}{a_2}},\dots, x_n^{\frac{1}{a_n}}] \ar[d]^{\iota'}\\
		S^+ \ar[r]^{\sigma'} \ar[d] & S^+ \ar[d]\\
		B  \ar[r]^{\sigma''} & C,
	}
\]
where $\sigma''$ denotes the bottom horizontal morphism.
Then
\[
	\sigma''((x_2^{a_2}+\dots+x_n^{a_n})^{\frac{a_1-1}{a_1}})\in (p,x_2^{\frac{b_2}{a_2}},\dots,x_n^{\frac{b_n}{a_n}})C.
\]
Hence,
\[
	(x_2^{b_2}+\dots+x_n^{b_n})^{\frac{a_1-1}{a_1}}\in (p,x_2^{\frac{b_2}{a_2}},\dots,x_n^{\frac{b_n}{a_n}})C.
\]
By Proposition \ref{BCM-regularity and brt} again, $\Z_p\llbracket x_1,\dots,x_n \rrbracket/(x_1^{b_1}+\dots +x_n^{b_n})$ is not BCM-regular.
\end{proof}
\section{The \texorpdfstring{$p$}{p}-th root formulae}
In this section, we present basic formulae used in later chapters.
Let $p$ be a prime number, $n$ be a positive integer and $S:=\Z_p\llbracket x_1,\dots,x_n\rrbracket$. We define a $\Z_p$-algebra endomorphism
\[
	\phi:S\to S; f\mapsto f(x_1^p,\dots,x_n^p),
\]
and set
\[
	\delta(f):=\frac{\phi(f)-f^p}{p}.
\]
For $S_\infty:=\Z_p\llbracket x_1,\dots,x_n \rrbracket[x_1^{1/p^\infty},\dots, x_n^{1/p^\infty}]$, we can naturally extend $\phi$ to an automorphism on $S_\infty$. When we consider $f^{1/p^i}$ for an element $f\in S$, we implicitly assume that $\{f^{1/p^i}\}_{i\ge 0}$ is a compatible system, i.e., $(f^{1/p^{i+1}})^p=f^{1/p^{i}}$.
\begin{lem} \label{p-root formula}
	Let $S=\Z_p\llbracket x_1,\dots,x_n\rrbracket$ and $f\in S$. Then, for any $e\ge 1$, there exists $\beta_e\in S^+$ such that 
\begin{align*}
	f^{\frac{1}{p}} &= \phi^{-1}(f)+\sum_{i=1}^{e} (-1)^{i-1}p^{\frac{1}{p}+\cdots+\frac{1}{p^i}} (\phi^{-i}(f))^{p^{i-1}-1}\phi^{-(i+1)}(\delta(f))\\
&+p^{\frac{1}{p}+\dots+\frac{1}{p^{e+1}}}\beta_e 
\end{align*}
in $S^+$.
\end{lem}
\begin{proof}
	If $p=2$, then we have
\begin{align*}
	(f^{\frac{1}{2}}-\phi^{-1}(f))^2 &=f-2f^{\frac{1}{2}}\phi^{-1}(f)+(\phi^{-1}(f))^2 \\
	&= -2\phi^{-1}(\delta(f))+2f^{\frac{1}{2}}(f^{\frac{1}{2}}-\phi^{-1}(f)).
\end{align*}
	If $p\ge 3$, then there exists $g\in S^+$ such that
\begin{align*}
	(f^{\frac{1}{p}}-\phi^{-1}(f))^p & = f-(\phi^{-1}(f))^p+p(f^{\frac{1}{p}}-\phi^{-1}(f))g\\
	& =p\phi^{-1}(\delta(f))+p(f^{\frac{1}{p}}-\phi^{-1}(f))g.
\end{align*}
Hence, in either case, $f^{\frac{1}{p}}-\phi^{-1}(f)\in p^{1/p}S^+$. Since
\begin{align*}
	(f^{\frac{1}{p}}-\phi^{-1}(f)-p^{\frac{1}{p}}\phi^{-2}(\delta(f)))^p &\equiv p\phi^{-1}(\delta(f))-p(\phi^{-2}(\delta(f)))^p \\
&\equiv 0 \pmod{p^{1+\frac{1}{p}}S^+},
\end{align*}
 we have
\[
f^{\frac{1}{p}}-\phi^{-1}(f)-p^{\frac{1}{p}}\phi^{-2}(\delta(f))\in p^{\frac{1}{p}+\frac{1}{p^2}}S^+.
\]
This shows the case where $e=1$.
Let
\[
	\alpha_e:=\phi^{-1}(f)+\sum_{i=1}^{e} (-1)^{i-1}p^{\frac{1}{p}+\dots+\frac{1}{p^i}} (\phi^{-i}(f))^{p^{i-1}-1}\phi^{-(i+1)}(\delta(f)).
\]
It is enough to show the following claim.
\begin{cl}
For $e \ge 1$,
\[
	(f^{\frac{1}{p}}-\alpha_e)^p \equiv -p(\phi^{-1}(f))^{p-1}(f^{\frac{1}{p}}-\alpha_{e-1}) \pmod {p^{1+\frac{2}{p}} S^+}
\]
and $f^{\frac{1}{p}}-\alpha_e\in p^{\frac{1}{p}+\dots+\frac{1}{p^{e+1}}}S^+$.
\end{cl}
\begin{clproof}
	We show the claim by induction on $e$. Suppose that $e=1$. Since
\[
(x-y)^p-(x^p-y^p)\equiv -py^{p-1}(x-y) \pmod {p(x-y)^2\Z[x,y]}
\]
in $\Z[x,y]$,
we obtain
	\begin{align*}
	(f^{\frac{1}{p}}-\alpha_1)^p&=(f^{\frac{1}{p}}-\phi^{-1}(f)-p^{\frac{1}{p}}\phi^{-2}(\delta(f)))^p \\
&\equiv f-(\phi^{-1}(f))^p-p(\phi^{-1}(f))^{p-1}(f^{\frac{1}{p}}-\phi^{-1}(f)) -p\phi^{-1}(\delta(f)) \\
&\equiv -p(\phi^{-1}(f))^{p-1}(f^{\frac{1}{p}}-\phi^{-1}(f)) \pmod{p^{1+\frac{2}{p}}S^+}.
	\end{align*}
Since $f^{\frac{1}{p}}-\phi^{-1}(f)\in p^{1/p}S^+$, we see that $f^{1/p}-\alpha_1\in p^{1/p+1/p^2}S^+$,
	which shows the case where $e=1$. Suppose that $e\ge 2$.
	By induction hypothesis, we see that
	\begin{align*}
		(f^{\frac{1}{p}}-\alpha_{e})^p &= \left(f^{\frac{1}{p}}-\alpha_{e-1}-(-1)^{e-1}p^{\frac{1}{p}+\dots+\frac{1}{p^e}}(\phi^{-e}(f))^{p^{e-1}-1}\phi^{-(e+1)}(\delta(f))\right)^p \\
		& \equiv -p(\phi^{-1}(f))^{p-1}(f^{\frac{1}{p}}-\alpha_{e-2}) \\ &+(-1)^{e-2}p^{1+\frac{1}{p}+\dots+\frac{1}{p^{e-1}}}(\phi^{-1}(f))^{p-1}(\phi^{-(e-1)}(f))^{p^{e-2}-1}\phi^{-e}(\delta(f)) \\
        &\pmod {p^{1+\frac{2}{p}}S^+} \\
		&= -p(\phi^{-1}(f))^{p-1}(f^{\frac{1}{p}}-\alpha_{e-1}).
	\end{align*}
	Since $f^{\frac{1}{p}}-\alpha_{e-1}\in p^{\frac{1}{p}+\dots+\frac{1}{p^e}}S^+$, we obtain $f^{\frac{1}{p}}-\alpha_e\in p^{\frac{1}{p}+\dots+\frac{1}{p^{e+1}}}S^+$.
\end{clproof}
\end{proof}

\begin{prop} \label{remark p-root formula}
	Let $S=\Z_p\llbracket x_1,\dots, x_n \rrbracket$ and $f\in S$. Suppose that $p, f$ is a regular sequence of $S$. Then there exist $\gamma_j\in S^+$ for $j=1,\dots, p$ such that 
\begin{align*}
	f^{\frac{1}{p}}&=\phi^{-1}(f)+\sum_{i=1}^{e} (-1)^{i-1}p^{\frac{1}{p}+\dots+\frac{1}{p^i}} (\phi^{-i}(f))^{p^{i-1}-1}\phi^{-(i+1)}(\delta(f)) \\
&+p^{\frac{1}{p}+\dots+\frac{1}{p^{e+1}}}(\phi^{-2}(f))^{p-1}\gamma_1+\sum_{j=2}^{p}p^{\frac{1}{p}+\frac{j}{p^2}}(\phi^{-2}(f))^{p-j}\gamma_j.
\end{align*}
\end{prop}
\begin{proof}
	If $p=2$, then
	\[
	(f^{\frac{1}{2}}-\phi^{-1}(f)-2^{\frac{1}{2}}\phi^{-2}(\delta(f)))^2 \equiv 2\phi^{-1}(f)(f^{\frac{1}{2}}-\phi^{-1}(f)) \pmod{4S^+}.
	\]
If $p\ge 3$, we have
\begin{align*}
	&(f^{\frac{1}{p}}-\phi^{-1}(f)-p^{\frac{1}{p}}\phi^{-2}(\delta(f)))^p \\
    &\equiv f-(\phi^{-1}(f))^p+\sum_{i=1}^{p-1}(-1)^{p-i}\binom{p}{i}f^{\frac{i}{p}}(\phi^{-1}(f))^{p-i}-p\phi^{-1}(\delta(f)) \\
& =\sum_{i=1}^{p-1}(-1)^{p-i}\binom{p}{i}f^{\frac{i}{p}}(\phi^{-1}(f))^{p-i}\\
& \pmod{p^2S^+}.
\end{align*}
Hence, in either case,
\begin{align*}
&(f^{\frac{1}{p}}-\phi^{-1}(f)-p^{\frac{1}{p}}\phi^{-2}(\delta(f)))^p \\
&\in \left(p^2, p(f^{\frac{1}{p}}-\phi^{-1}(f))^i(\phi^{-1}(f))^{p-i}\middle| 1\le i \le p-1\right)S^+ \\
& \subseteq \left(p^{1+\frac{i}{p}}(\phi^{-1}(f))^{p-i}\middle| 1\le i \le p\right)S^+.
\end{align*}
Hence, there exist $\gamma'_i\in S^+$ for $i=1,\dots, p$ such that
\[
	f^{\frac{1}{p}}=\phi^{-1}(f)+p^{\frac{1}{p}}\phi^{-2}(\delta(f))+\sum_{i=1}^{p}p^{\frac{1}{p}+\frac{i}{p^2}}(\phi^{-2}(f))^{p-i}\gamma'_i.
\]
By Lemma \ref{p-root formula}, for any $e\ge 1$, there exists $\beta_e\in S^+$ such that 
\begin{align*}
	&\sum_{i=2}^{e} (-1)^{i-1}p^{\frac{1}{p}+\dots+\frac{1}{p^i}} (\phi^{-i}(f))^{p^{i-1}-1}\phi^{-(i+1)}(\delta(f))+p^{\frac{1}{p}+\dots+\frac{1}{p^{e+1}}}\beta_e\\
&=\sum_{j=1}^{p}p^{\frac{1}{p}+\frac{j}{p^2}}(\phi^{-2}(f))^{p-j}\gamma'_j.
\end{align*}
Hence, 
\begin{align*}
p^{\frac{1}{p}+\dots+\frac{1}{p^{e+1}}}\beta_e -\sum_{j=2}^{p}p^{\frac{1}{p}+\frac{j}{p^2}}(\phi^{-2}(f))^{p-j}\gamma'_j\in (p,(\phi^{-2}(f))^{p-1})S^+.
\end{align*}
Since the $p$-adic completion $\widehat{S^+}^p$ of $S^+$ is a BCM $S$-algebra by \cite[Corollary 5.17]{Bhatt21} and $p, f$ is a part of a  system of parameters for $S$, there exists $\tilde{\gamma_1}\in \widehat{S^+}^p$ such that
\begin{align*}
	p^{\frac{1}{p}+\dots+\frac{1}{p^{e+1}}}\beta_e =p^{\frac{1}{p}+\dots+\frac{1}{p^{e+1}}}(\phi^{-2}(f))^{p-1}\tilde{\gamma_1} + \sum_{j=2}^{p}p^{\frac{1}{p}+\frac{j}{p^2}}(\phi^{-2}(f))^{p-j}\gamma'_j
\end{align*}
in $\widehat{S^+}^p$. Let $\gamma_j=\gamma'_j$ for $2\le j \le p-1$. We can take $\gamma_1, \gamma_p \in S^+$ such that
\begin{align*}
	p^{\frac{1}{p}+\dots+\frac{1}{p^{e+1}}}\beta_e =p^{\frac{1}{p}+\dots+\frac{1}{p^{e+1}}}(\phi^{-2}(f))^{p-1}\gamma_1 + \sum_{j=2}^{p}p^{\frac{1}{p}+\frac{j}{p^2}}(\phi^{-2}(f))^{p-j}\gamma_j
\end{align*}
in $S^+$, which completes the proof.
\end{proof}

\begin{thm} \label{p^e-root formula}
		Let $S=\Z_p\llbracket x_1,\dots,x_n\rrbracket$ and $f\in S$. Suppose that $p, f$ is a regular sequence of $S$. Then, for any $e,e'\ge 1$ and for $j=1,\dots,p$, there exists $\gamma_{e,j}\in S^+$ such that 
\begin{align*}
	f^{\frac{1}{p^{e'}}}= &\phi^{-e'}(f)+\sum_{i=1}^{e} (-1)^{i-1}p^{\frac{1}{p^{e'}}+\dots+\frac{1}{p^{e'+i-1}}} (\phi^{-(e'+i-1)}(f))^{p^{i-1}-1}\phi^{-(e'+i)}(\delta(f)) \\
&+p^{\frac{1}{p^{e'}}+\dots+\frac{1}{p^{e'+e}}}(\phi^{-(e'+1)}(f))^{p-1}\gamma_{e,1} + \sum_{j=2}^{p}p^{\frac{1}{p^{e'}}+\frac{j}{p^{e'+1}}}(\phi^{-(e'+1)}(f))^{p-j}\gamma_{e,j}
\end{align*}
in $S^+$.
\end{thm}
\begin{proof}
	Let 
\begin{align*}
	&\alpha_{e,e'} \\&:= \phi^{-e'}(f)+\sum_{i=1}^{e} (-1)^{i-1}p^{\frac{1}{p^{e'}}+\dots+\frac{1}{p^{e'+i-1}}} (\phi^{-(e'+i-1)}(f))^{p^{i-1}-1}\phi^{-(e'+i)}(\delta(f)) \\
&+p^{\frac{1}{p^{e'}}+\dots+\frac{1}{p^{e'+e}}}(\phi^{-(e'+1)}(f))^{p-1}(\gamma_{1}^{(e)})^{\frac{1}{p^{e'-1}}} + \sum_{j=2}^{p}p^{\frac{1}{p^{e'}}+\frac{j}{p^{e'+1}}}(\phi^{-(e'+1)}(f))^{p-j}(\gamma_{j}^{(e)})^{\frac{1}{p^{e'-1}}},
\end{align*}
where $\gamma_j^{(e)}$ is defined as $\gamma_j$ in Proposition \ref{remark p-root formula}.
For any $g\in S$, we see that $g^{p^{e'-1}}-\phi^{e'-1}(g)\in pS$. Hence, $(\phi^{-(e'-1)}(g))^{p^{e'-1}}-g\in pS^+$. Therefore, we have
\begin{align*}
	&(f^{\frac{1}{p^{e'}}}-\alpha_{e,e'})^{p^{e'-1}} \\
	&\equiv f^{\frac{1}{p}}-\phi^{-1}(f)-\sum_{i=1}^{e} (-1)^{i-1}p^{\frac{1}{p}+\dots+\frac{1}{p^i}} (\phi^{-i}(f))^{p^{i-1}-1}\phi^{-(i+1)}(\delta(f))\\
&\quad -p^{\frac{1}{p}+\dots+\frac{1}{p^{e+1}}}(\phi^{-2}(f))^{p-1}\gamma_1^{(e)}-\sum_{j=2}^{p}p^{\frac{1}{p}+\frac{j}{p^2}}(\phi^{-2}(f))^{p-j}\gamma_j^{(e)} \pmod {pS^+}\\
& = 0 .
\end{align*}
Hence, $f^{\frac{1}{p^{e'}}}-\alpha_{e,e'}\in p^{\frac{1}{p^{e'-1}}}S^+$, which completes the proof.
\end{proof}

\section{Fermat-type singularities in mixed characteristic}
In this section, we study BCM-regularity of Fermat-type hypersurfaces in mixed characteristic.
\begin{setting}\label{setting Fermat type}
	Let $p$ be a prime number, $n\ge2$ and $d$ be positive integers, and $R:=\Z_p\llbracket x_0,\dots, x_n\rrbracket/(x_0^d+\dots+x_n^d)$.
\end{setting}
\begin{ques}
	When is $R$ BCM-regular?
\end{ques}
\subsection{Positive results}
This subsection is devoted to cases in which Fermat-type hypersurfaces are BCM-regular.

	First, we prove a lemma that will be needed later on.
\begin{lem}
	Let $p$ be a prime number, $n,d_1,\dots,d_n$ be positive integers, let $S:=\Z_p\llbracket x_1,\dots,x_n \rrbracket$ and let $f:=x_1^{d_1}+\dots+x_n^{d_n}$. Suppose that $s$ and $t$ are positive integers such that $0\le s,t< p$. Let $C_{m_1,\dots,m_n}$ denote the coefficient of $x_1^{m_1d_1}\cdots x_n^{m_nd_n}$ in $(\delta(f))^sf^t$ for non-negative integers $m_1,\dots,m_n$. If $m_1+\dots+m_n=sp+t$ and $m_i< p$ for any $i$, then $C_{m_1,\dots,m_n}$ is coprime to $p$.
\end{lem}
\begin{proof}
	We have
	\begin{align*}
		f^{sp+t} &=(\phi(f)-p\delta(f))^{s}f^t\\
			&\equiv (-p)^s(\delta(f))^sf^t \pmod{(x_1^{pd_1},\dots,x_n^{pd_n})}.
	\end{align*}
	and
	\begin{align*}
		f^{sp+t} &= \frac{(sp+t)!}{m_1!\cdots m_n!}x_1^{m_1d_1}\cdots x_n^{m_nd_n} +(\text{other terms}).
	\end{align*}
	Therefore,
	\[
		C_{m_1,\dots,m_n}=\frac{(sp+t)!}{(-p)^sm_1!\dots m_n!},
	\]
	which is coprime to $p$.
\end{proof}

\begin{prop} \label{Fermat type <=p}
	Let the notation be as in Setting \ref{setting Fermat type}. If $d\le n\le p$, then $R$ is BCM-regular.
\end{prop}
\begin{proof}
	We may assume that $d=n$. It is enough to show that $\Z_p\llbracket x_0,\dots, x_n\rrbracket /(x_0^p+x_1^n+\dots +x_n^n)$ is BCM-regular. Let $f:=x_1^n+\dots+x_n^n$. By Lemma \ref{p-root formula}, there exists $\beta_2 \in S^+$ such that
	\[
		f^{\frac{1}{p}}=\phi^{-1}(f)+p^{\frac{1}{p}}\phi^{-2}(\delta(f))+p^{\frac{1}{p}+\frac{1}{p^2}}\beta_2.
	\]
	Let $Q$, $r$  be non-negative integers such that $p-1=nQ+r$ and $0\le r <n$. Since $p-r=nQ+1$, we see that
	\[
		f^{p-r} \in  (x_1^p,\dots,x_n^p)S.
	\]
	Hence,
	\[
		f^{\frac{p-1}{p}} \equiv p^{\frac{r}{p}}\binom{p-1}{r}(\phi^{-1}(f))^{p-r-1}(\phi^{-2}(\delta(f)))^r \pmod{(p^{\frac{r}{p}+\frac{1}{p^2}}, x_1,\dots,x_n)S^+}.
	\]
Hence, it is enough to show that
\[
	f^{p(p-r-1)}(\delta(f))^r \notin (p,x_1^{p^2},\dots, x_n^{p^2})S.
\]
Since $p-r-1=nQ$, we obtain
\[
	f^{p(p-r-1)} \equiv \frac{p-r-1}{(n!)^Q}x_1^{npQ}\dots x_n^{npQ} \pmod{(p,x_1^{p^2},\dots,x_n^{p^2})S}.
\]

Let $Q'$, $r'$ be non-negative integers such that $pr=nQ'+r'$ and $0\le r' < n$. Note that $Q'+1<p$ if $r'>0$. Indeed, $n(Q'+1)< p(r+1)\le pn$ if $r'>0$. By the above lemma, there exists an integer $C$ coprime to $p$ such that $(\delta(f))^r=Cx_1^{n(Q'+1)}\cdots x_{r'}^{n(Q'+1)} x_{r'+1}^{nQ'}\cdots x_{n}^{nQ'} + (\text{other terms})$.
Hence, $f^{p(p-r-1)}\delta(f)^r\notin (p,x_1^{p^2},\dots, x_n^{p^2})S$, which completes the proof.
\end{proof}

\begin{prop} \label{Fermat type p^2-1/a}
	Let the notation be as in Setting \ref{setting Fermat type}. Let $a$ be a positive integer such that $1\le a \le p-1$. Suppose that $d\le n$ and
\[
	n=\lfloor \frac{p^2-1}{a}\rfloor.
\]
Then $R$ is BCM-regular.
\end{prop}
\begin{proof}
	We may assume that $d=n$.
	Let $S=\Z_p\llbracket x_1,\dots,x_n\rrbracket$, $f=x_1^n+\dots+x_n^n$ and $r:=an-(p^2-p)$.
	By Proposition \ref{remark p-root formula}, we have
	\begin{align*}
		f^{\frac{1}{p}} \equiv p^{\frac{1}{p}}\phi^{-2}(\delta(f))\pmod{(p^{\frac{1}{p}+\frac{1}{p^2}},x_1,\dots,x_n)S^+}.
	\end{align*}
	Hence, it follows that
	\begin{align*}
		f^{\frac{p-1}{p}} &\equiv p^{\frac{p-1}{p}}(\phi^{-2}(\delta(f)))^{p-1} \pmod {(p^{\frac{p-1}{p}+\frac{1}{p^2}},x_1,\dots,x_n)S^+}.
	\end{align*}
	By Theorem \ref{p^e-root formula}, we see that
\begin{align*}
	f^{\frac{1}{p^2}} &\equiv \phi^{-2}(f) \pmod {p^{\frac{1}{p^2}}S^+}.
\end{align*}
Hence, we obtain
\begin{align*}
	f^{\frac{p^2-a}{p^2}} &= f^{\frac{p-1}{p}}f^{\frac{p-a}{p^2}} \\
 &\equiv p^{\frac{p-1}{p}}(\phi^{-2}(\delta(f)))^{p-1}(\phi^{-2}(f))^{p-a} \pmod{(p^{\frac{p-1}{p}+\frac{1}{p^2}},x_1,\dots,x_n)S^+}.
\end{align*}

There exists an integer $C$ coprime to $p$ such that
\[
	f^{p-a}(\delta(f))^{p-1}=Cx_1^{an}\cdots x_{n-1}^{an}x_{n}^{(p-r)n}+ (\text{other terms}).
\] 
Hence, 
\[
	f^{\frac{p^2-a}{p^2}}\notin (p,x_1,\dots,x_n)B
\]
for any perfectoid BCM $S^+$-algbera $B$. Since
\[
	\frac{p^2-a}{p^2}>\frac{n-1}{n},
\]
we obtain 
\[
	f^{\frac{n-1}{n}}\notin (p,x_1,\dots,x_n)B,
\]
which completes the proof.
\end{proof}

\begin{prop} \label{Fermat type an=p^2-a-1}
	Let the notation be as in Setting \ref{setting Fermat type}. Suppose that $p\ge 3$.  Let $a$ be a positive integer such that $1\le a \le p-1$ and suppose that $an=p^2-a-1$. Then $R$ is BCM-regular.
\end{prop}
\begin{proof}
	Let $S=\Z_p\llbracket x_1,\dots,x_n \rrbracket$ and $f=x_1^n+\dots+x_n^n$. 
By Theorem \ref{p^e-root formula},
\begin{align*}
	f^{\frac{1}{p}}&\equiv \phi^{-1}(f)+p^{\frac{1}{p}}\phi^{-2}(\delta(f))-p^{\frac{1}{p}+\frac{1}{p^2}}(\phi^{-2}(f))^{p-1}\phi^{-3}(\delta(f))+p^{\frac{1}{p}+\frac{1}{p^2}+\frac{1}{p^3}}(\phi^{-2}(f))^{p-1}\beta \\
&\pmod{(p^{\frac{1}{p}+\frac{2}{p^2}},x_1,\dots,x_n)S^+}
\end{align*}
for some $\beta\in S^+$. Hence,
\begin{align*}
	f^{\frac{p-1}{p}}& \equiv p^{\frac{p-1}{p}}(\phi^{-2}(\delta(f)))^{p-1}+p^{\frac{p-1}{p}+\frac{1}{p^2}}(\phi^{-2}(f))^{p-1}(\phi^{-2}(\delta(f)))^{p-2}\phi^{-3}(\delta(f))\\
&-p^{\frac{p-1}{p}+\frac{1}{p^2}+\frac{1}{p^3}}\phi^{-2}(f)^{p-1}(\phi^{-2}(\delta(f)))^{p-2}\beta \\
&\pmod{(p^{\frac{p-1}{p}+\frac{2}{p^2}},x_1,\dots,x_n)S^+}.
\end{align*}
Similarly, we have
\begin{align*}
	f^{\frac{1}{p^2}} \equiv \phi^{-2}(f)+p^{\frac{1}{p^2}}\phi^{-3}(\delta(f))+p^{\frac{1}{p^2}+\frac{1}{p^3}}(\phi^{-3}(f))^{p-1}\beta_2 \pmod{(p^{\frac{1}{p^2}+\frac{2}{p^3}},x_1,\dots,x_n)S^+}
\end{align*}
for some $\beta_2\in S^+$. First, suppose that $(ap+1)n<p^3$. We have
\begin{align*}
	f^{\frac{p^2-a}{p^2}} &= f^{\frac{p-1}{p}}f^{\frac{p-a}{p^2}}\\
&\equiv -ap^{\frac{p-1}{p}+\frac{1}{p^2}} (\phi^{-2}(\delta(f)))^{p-1}(\phi^{-2}(f))^{p-a-1}\phi^{-3}(\delta(f)) \\
&\pmod {(p^{\frac{p-1}{p}+\frac{1}{p^2}+\frac{1}{p^3}},x_1,\dots,x_n)S^+}.
\end{align*}
There exists an integer $C_1$ coprime to $p$ such that
\begin{align*}
	f^{\frac{p^2-a}{p^2}} &\equiv -aC_1p^{\frac{p-1}{p}+\frac{1}{p^2}} x_1^{\frac{an}{p^2}}\dots x_n^{\frac{an}{p^2}} \phi^{-3}(\delta(f)) \\
&\equiv a((p-1)!)C_1p^{\frac{p-1}{p}+\frac{1}{p^2}} x_1^{\frac{(ap+1)n}{p^3}}\cdots x_p^{\frac{(ap+1)n}{p^3}}x_{p+1}^{\frac{an}{p^2}}\cdots x_n^{\frac{an}{p^2}} + (\text{other terms}) \\
&\pmod {(p^{\frac{p-1}{p}+\frac{1}{p^2}+\frac{1}{p^3}},x_1,\dots,x_n)S^+}.
\end{align*}
Hence, $f^{\frac{p^2-a}{p^2}}\notin (p,x_1,\dots,x_n)B$ for any perfectoid BCM $S^+$-algebra. Since $an<p^2$, we have $(n-1)/n<(p^2-a)/p^2$. Therefore,
\[
	f^{\frac{n-1}{n}}\notin (p,x_1,\dots,x_n)B,
\]
which completes the proof in this case.
Next, suppose that $(ap+1)n\ge p^3$. We have
\begin{align*}
	f^{\frac{p-1}{p}+\frac{p-a-1}{p^2}} &= f^{\frac{p-1}{p}}f^{\frac{p-a-1}{p^2}}\\
&\equiv p^{\frac{p-1}{p}} (\phi^{-2}(\delta(f)))^{p-1}(\phi^{-2}(f))^{p-a-1} \pmod {(p^{\frac{p-1}{p}+\frac{1}{p^2}},x_1,\dots,x_n)S^+}.
\end{align*}
Hence,
\[
	f^{\frac{p-1}{p}+\frac{p-a-1}{p^2}}\equiv p^{\frac{p-1}{p}}C_1 x_1^{\frac{an}{p^2}}\dots x_n^{\frac{an}{p^2}} \pmod {(p^{\frac{p-1}{p}+\frac{1}{p^2}},x_1,\dots,x_n)S^+}.
\]

We also have
\[
	f^{\frac{1}{p^3}} \equiv \phi^{-3}(f)+p^{\frac{1}{p^3}}\phi^{-4}(\delta(f)) \pmod{p^{\frac{1}{p^3}+\frac{1}{p^4}}S^+}
\]
and
\[
	f^{\frac{1}{p^4}} \equiv \phi^{-4}(f) \pmod{p^{\frac{1}{p^4}}S^+}.
\]
Since $(ap+1)n\ge p^3$, we see that 
\[
	x_1^{\frac{an}{p^2}}\cdots x_n^{\frac{an}{p^2}}\phi^{-3}(f) \in (x_1,\dots,x_n)S^+.
\]
Hence,
\begin{align*}
	f^{\frac{p-1}{p}+\frac{p-a-1}{p^2}+\frac{p-1}{p^3}+\frac{p-1}{p^4}} &\equiv p^{\frac{p-1}{p}+\frac{p-1}{p^3}}C_1x_1^{\frac{an}{p^2}}\cdots x_n^{\frac{an}{p^2}}(\phi^{-4}(\delta(f)))^{p-1}(\phi^{-4}(f))^{p-1} \\
&\pmod{(p^{\frac{p-1}{p}+\frac{p-1}{p^3}+\frac{1}{p^4}},x_1,\dots,x_n)S^+}.
\end{align*}
Therefore, there exists an integer $C_2$ coprime to $p$ such that
\begin{align*}
	&f^{\frac{p-1}{p}+\frac{p-a-1}{p^2}+\frac{p-1}{p^3}+\frac{p-1}{p^4}} \equiv p^{\frac{p-1}{p}+\frac{p-1}{p^3}} C_2 x_1^{\frac{(ap^2+a+1)n}{p^4}}\cdots x_a^{\frac{(ap^2+a+1)n}{p^4}}x_{a+1}^{\frac{(ap^2+a)n}{p^4}}\cdots x_n^{\frac{(ap^2+a)n}{p^4}}\\
& \pmod {(p^{\frac{p-1}{p}+\frac{p-1}{p^3}+\frac{1}{p^4}},x_1^{\frac{(ap^2+a+2)n}{p^4}},\dots,x_a^{\frac{(ap^2+a+2)n}{p^4}},x_{a+1}^{\frac{(ap^2+a+1)n}{p^4}}, \dots, x_n^{\frac{(ap^2+a+1)n}{p^4}},\\
& \qquad \qquad x_1,\dots ,x_a)S^+}.
\end{align*}
Since
\[
	(ap^2+a+1)n=(p^2+1)(p^2-a-1)+n=p^4-ap^2-a-1+n<p^4,
\]
it follows that 
\[
	f^{\frac{p-1}{p}+\frac{p-a-1}{p^2}+\frac{p-1}{p^3}+\frac{p-1}{p^4}}\notin (p,x_1,\dots,x_n)B
\]
for any perfectoid BCM $S^+$-algebra $B$. Since
\[
	\frac{p-1}{p}+\frac{p-a-1}{p^2}+\frac{p-1}{p^3}+\frac{p-1}{p^4}>\frac{n-1}{n},
\]
we obtain
\[
	f^{\frac{n-1}{n}}\notin (p,x_1,\dots,x_n)B,
\]
which completes the proof.
\end{proof}

For the next proposition, we need the following more refined version of Proposition \ref{remark p-root formula}.
\begin{lem} \label{refined p-root formula}
	Let $p\ge 3$ be a prime number. Let $S:=\Z_p \llbracket x_1,\dots,x_n \rrbracket$ and $f\in S$. Then there exist $\gamma_{i,j}\in S^+$ such that	

\begin{align*}
&f^{\frac{1}{p}}\\
&= \phi^{-1}(f)+p^{\frac{1}{p}}\phi^{-2}(\delta(f))-\sum_{i=1}^{p-1}\frac{1}{p}\binom{p}{i}p^{\frac{1}{p}+\frac{i}{p^2}}(\phi^{-3}(\delta(f)))^{i}(\phi^{-2}(f))^{p-i}\\
& \quad +\sum_{i=1}^{p-1}\sum_{j=1}^{p-1}\frac{i}{p^2}\binom{p}{i}\binom{p}{j}p^{\frac{1}{p}+\frac{i}{p^2}+\frac{j}{p^3}}(\phi^{-2}(f))^{p-i}(\phi^{-3}(f))^{p-j}(\phi^{-3}(\delta(f)))^{i-1}(\phi^{-4}(\delta(f)))^{j}\\
& \quad -\sum_{i=1}^{p-1}\sum_{j=1}^{p-1}\frac{ij}{p^2}\binom{p}{i}\binom{p}{j}p^{\frac{1}{p}+\frac{i}{p^2}+\frac{j}{p^3}+\frac{1}{p^4}}\\
& \quad \cdot (\phi^{-2}(f))^{p-i}\phi^{-3}(f^{p-j}(\delta(f))^{i-1})\phi^{-4}(f^{p-1}(\delta(f))^{j-1})\phi^{-5}(\delta(f))\\
& \quad +\sum_{i=1}^{p-1}\sum_{j=1}^{p-1}p^{\frac{1}{p}+\frac{i}{p^2}+\frac{j}{p^3}+\frac{1}{p^4}+\frac{1}{p^5}}(\phi^{-2}(f))^{p-i}(\phi^{-3}(f))^{p-j}\gamma_{i,j}\\
&\quad +\sum_{i=1}^{p-1}p^{\frac{1}{p}+\frac{i+1}{p^2}}(\phi^{-2}(f))^{p-i}\gamma_{i,p} + p^{\frac{2}{p}}\gamma_{p,p}.
	\end{align*}
\end{lem}
\begin{proof}
	First, we show the following claim.
	\begin{cl}
		The equation
		$(x-y)^p=x^p-y^p-\sum_{i=1}^{p-1}\binom{p}{i}(x-y)^iy^{p-i}$
		holds in $\Z[x,y]$.
	\end{cl}
	\begin{clproof}
		We can write
		\[
			(x-y)^p=x^p+\sum_{i=0}^{p-1}C_i(x-y)^iy^{p-i}
		\]
		for some integers $C_i$. For $0\le n\le p-1$, differentiating both sides $n$ times with respect to $x$ and then evaluating at $x=y=1$, we obtain
		\[
			0=\frac{p!}{(p-n)!}+n!C_n.
		\]
		Therefore,
		\[
			C_n=-\binom{p}{n}.
		\]
	\end{clproof}
	By the above claim, we have
	\begin{align} \label{equation 1}
		(f^{\frac{1}{p}}-\phi^{-1}(f))^p &= p\phi^{-1}(\delta(f))-\sum_{i=1}^{p-1}\binom{p}{i}(f^{\frac{1}{p}}-\phi^{-1}(f))^{i}(\phi^{-1}(f))^{p-i}.
	\end{align}
	By Lemma \ref{p-root formula}, there exists $\beta\in S^+$ such that
\begin{align} \label{equation 2}
		f^{\frac{1}{p}}= \phi^{-1}(f)+p^{\frac{1}{p}}\phi^{-2}(\delta(f))-p^{\frac{1}{p}+\frac{1}{p^2}}(\phi^{-2}(f))^{p-1}\phi^{-3}(\delta(f))+p^{\frac{1}{p}+\frac{1}{p^2}+\frac{1}{p^3}}\beta.
	\end{align}
	Substituting \ref{equation 2} into \ref{equation 1}, we obtain
	\begin{align*}
		(f^{\frac{1}{p}}-\phi^{-1}(f))^p &= p\phi^{-1}(\delta(f))-\sum_{i=1}^{p-1}\binom{p}{i}p^{\frac{i}{p}}(\phi^{-2}(\delta(f)))^{i}(\phi^{-1}(f))^{p-i}\\
&\quad +\sum_{i=1}^{p-1}i\binom{p}{i}p^{\frac{i}{p}+\frac{1}{p^2}}(\phi^{-1}(f))^{p-i}(\phi^{-2}(f))^{p-1}(\phi^{-2}(\delta(f)))^{i-1}\phi^{-3}(\delta(f))\\
& \quad +\sum_{i=1}^{p-1}p^{1+\frac{i}{p}+\frac{1}{p^2}+\frac{1}{p^3}}(\phi^{-1}(f))^{p-i}\beta_{i}
\end{align*}
	for some $\beta_{i} \in S^+$.
	Therefore, there exist $\gamma_{i}\in S^+$ such that
	\begin{align*} \label{equation 3}
	\stepcounter{equation}\tag{\theequation}
		f^{\frac{1}{p}}&=\phi^{-1}(f)+p^{\frac{1}{p}}\phi^{-2}(\delta(f))-\sum_{i=1}^{p-1}\frac{1}{p}\binom{p}{i}p^{\frac{1}{p}+\frac{i}{p^2}}(\phi^{-3}(\delta(f)))^{i}(\phi^{-2}(f))^{p-i}\\
&\quad +\sum_{i=1}^{p-1}\frac{i}{p}\binom{p}{i}p^{\frac{1}{p}+\frac{i}{p^2}+\frac{1}{p^3}}(\phi^{-2}(f))^{p-i}(\phi^{-3}(f))^{p-1}(\phi^{-3}(\delta(f)))^{i-1}\phi^{-4}(\delta(f))\\
& \quad +\sum_{i=1}^{p-1}p^{\frac{1}{p}+\frac{i}{p^2}+\frac{1}{p^3}+\frac{1}{p^4}}(\phi^{-2}(f))^{p-i}\gamma_{i}+p^{\frac{2}{p}}\gamma_p.
	\end{align*}
Substituting \ref{equation 3} into \ref{equation 1}, there exist $\beta_{i,j} \in S^+$ such that
	\begin{align*}
		&(f^{\frac{1}{p}}-\phi^{-1}(f))^p \\ 
&= p\phi^{-1}(\delta(f))-\sum_{i=1}^{p-1}\binom{p}{i}p^{\frac{i}{p}}(\phi^{-2}(\delta(f)))^{i}(\phi^{-1}(f))^{p-i}\\
&\quad +\sum_{i=1}^{p-1}\sum_{j=1}^{p-1}\frac{i}{p}\binom{p}{i}\binom{p}{j}p^{\frac{i}{p}+\frac{j}{p^2}}(\phi^{-1}(f))^{p-i}(\phi^{-2}(f))^{p-j}(\phi^{-2}(\delta(f)))^{i-1}(\phi^{-3}(\delta(f)))^j\\
& \quad -\sum_{i=1}^{p-1}\sum_{j=1}^{p-1}\frac{ij}{p^2}\binom{p}{i}\binom{p}{j}p^{\frac{i}{p}+\frac{j}{p^2}+\frac{1}{p^3}}\\
& \quad \cdot (\phi^{-1}(f))^{p-i}\phi^{-2}(f^{p-j}(\delta(f))^{i-1}))\phi^{-3}(f^{p-1}(\delta(f))^{j-1})\phi^{-4}(\delta(f))\\
&\quad +\sum_{i=1}^{p-1}\sum_{j=1}^{p-1}p^{1+\frac{i}{p}+\frac{j}{p^2}+\frac{1}{p^3}+\frac{1}{p^4}}(\phi^{-1}(f))^{p-i}(\phi^{-2}(f))^{p-j}\beta_{i,j}+\sum_{i=1}^{p-1} p^{1+\frac{i+1}{p}}(\phi^{-1}(f))^{p-i}\beta_{i,p}.
	\end{align*}
	Therefore, the statement follows from the same argument as above.
\end{proof}
\begin{prop} \label{Fermat type an=p^2-p}
	Let the notation be as in Setting \ref{setting Fermat type}. Let $a$ be a positive integer such that $1\le a \le p-1$. Suppose that one of the following holds:
\begin{enumerate}
	\item $an=p^2-p$.
	\item $an=p^2-p-1$.
\end{enumerate}
Then $R$ is BCM-regular.
\end{prop}
\begin{proof}
	We may assume that $d=n$. Let $S:=\Z_p\llbracket x_1,\dots,x_n\rrbracket$ and $f=x_1^n+\dots+x_n^n$. By Lemma \ref{refined p-root formula},
there exist $\gamma_{i,j}\in S^+$ such that
	\begin{align*}
		&f^{\frac{1}{p}}\\
&=\phi^{-1}(f)+p^{\frac{1}{p}}\phi^{-2}(\delta(f))-\sum_{i=1}^{p-1}\frac{1}{p}\binom{p}{i}p^{\frac{1}{p}+\frac{i}{p^2}}(\phi^{-3}(\delta(f)))^{i}(\phi^{-2}(f))^{p-i}\\
&\quad +\sum_{i=1}^{p-1}\sum_{j=1}^{p-1}\frac{i}{p^2}\binom{p}{i}\binom{p}{j}p^{\frac{1}{p}+\frac{i}{p^2}+\frac{j}{p^3}}(\phi^{-2}(f))^{p-i}(\phi^{-3}(f))^{p-j}(\phi^{-3}(\delta(f)))^{i-1}(\phi^{-4}(\delta(f)))^{j}\\
& \quad -\sum_{i=1}^{p-1}\sum_{j=1}^{p-1}\frac{ij}{p^2}\binom{p}{i}\binom{p}{j}p^{\frac{1}{p}+\frac{i}{p^2}+\frac{j}{p^3}+\frac{1}{p^4}}\\
& \quad \cdot (\phi^{-2}(f))^{p-i}\phi^{-3}(f^{p-j}(\delta(f))^{i-1})\phi^{-4}(f^{p-1}(\delta(f))^{j-1})\phi^{-5}(\delta(f))\\
&\quad +\sum_{i=1}^{p-1}\sum_{j=1}^{p-1}p^{\frac{1}{p}+\frac{i}{p^2}+\frac{j}{p^3}+\frac{1}{p^4}+\frac{1}{p^5}}(\phi^{-2}(f))^{p-i}(\phi^{-3}(f))^{p-j}\gamma_{i,j}\\
& \quad +\sum_{i=1}^{p-1}p^{\frac{1}{p}+\frac{i+1}{p^2}}(\phi^{-2}(f))^{p-i}\gamma_{i,p} + p^{\frac{2}{p}}\gamma_{p,p}.
	\end{align*}
Suppose that $an=p^2-p$. Then
\begin{align*}
	f^{\frac{p-1}{p}}&\equiv p^{\frac{p-1}{p}}(\phi^{-2}(\delta(f)))^{p-1}+\sum_{i=1}^{p-1}p^{\frac{p-1}{p}+\frac{i}{p^2}}(\phi^{-2}(f))^{p-i}(\phi^{-2}(\delta(f)))^{p-2}\gamma_i \\
    &\pmod {(p,x_1,\dots,x_n)S^+}
\end{align*} for some $\gamma_i \in S^+$. Hence, we obtain
\begin{align*}
	f^{\frac{p-1}{p}} &\equiv p^{\frac{p-1}{p}}C_0x_1^{\frac{an}{p^2}}\cdots x_n^{\frac{an}{p^2}}+\sum_{i=1}^{p-1}p^{\frac{p-1}{p}+\frac{i}{p^2}}(\phi^{-2}(f))^{p-i}(\phi^{-2}(\delta(f)))^{p-2}\gamma_i \\
    & \pmod {(p,x_1,\dots,x_n)S^+}
\end{align*}
for some $C_0\in \Z\setminus p\Z$.
We also have
\[
	f^{\frac{1}{p^2}}\equiv \phi^{-2}(f)+p^{\frac{1}{p^2}}\phi^{-3}(\delta(f)) \pmod {p^{\frac{1}{p^2}+\frac{1}{p^3}}S^+}.
\]
Therefore,
\[
f^{\frac{p^2-a}{p^2}}=p^{\frac{p-1}{p}+\frac{p-a}{p^2}}C_0x_1^{\frac{an}{p^2}}\dots x_n^{\frac{an}{p^2}}(\phi^{-3}(\delta(f)))^{p-a} \pmod{(p^{\frac{p-1}{p}+\frac{p-a}{p^2}+\frac{1}{p^3}},x_1,\dots,x_n)S^+}.
\]
Let $r, Q$ be non-negative integers such that $p^2-ap=nQ+r$ and $0\le r <n$. If $a> 1$, then $Q< a$, and if $a=1$, then $r=0$.
We have
\begin{align*}
	&f^{\frac{p^2-a}{p^2}} \equiv p^{\frac{p-1}{p}+\frac{p-a}{p^2}}C_1x_1^{\frac{(ap+Q+1)n}{p^3}}\cdots x_r^{\frac{(ap+Q+1)n}{p^3}}x_{r+1}^{\frac{(ap+Q)n}{p^3}}\cdots x_n^{\frac{(ap+Q)n}{p^3}} \\
&\pmod{(p^{\frac{p-1}{p}+\frac{p-a}{p^2}+\frac{1}{p^3}},x_1^{\frac{(ap+Q+2)n}{p^3}},\dots,x_r^{\frac{(ap+Q+2)n}{p^3}},x_{r+1}^{\frac{(ap+Q+1)n}{p^3}},\dots, x_n^{\frac{(ap+Q+1)n}{p^3}},\\
& \qquad \qquad x_1,\dots,x_n)S^+}
\end{align*} for some $C_1\in \Z\setminus p\Z$. Note that, if $r>0$, then 
\[
(ap+Q+1)n\le an(p+1)= (p^2-p)(p+1)=p^3-p<p^3.
\]
Hence,  
\[
f^{\frac{p^2-a}{p^2}}\notin (p,x_1,\dots,x_n)B
\]
for any perfectoid BCM $S^+$-algebra $B$. Since 
\[
\frac{n-1}{n}<\frac{p^2-a}{p^2},
\]
we have $f^{(n-1)/n}\notin (p,x_1,\dots,x_n)B$. Hence, $R$ is BCM-regular. For (2), suppose that $an=p^2-p-1$. By a similar argument, there exist $\gamma'_{i,j}\in S^+$ such that 
\begin{align*}
f^{\frac{p-1}{p}} &\equiv \sum_{i=1}^{p-1}\frac{1}{p}\binom{p}{i}p^{\frac{p-1}{p}+\frac{i}{p^2}}(\phi^{-2}(\delta(f)))^{p-2}(\phi^{-3}(\delta(f)))^i(\phi^{-2}(f))^{p-i} \\
&\quad -\sum_{i=1}^{p-1}\sum_{j=1}^{p-1}\biggl(\frac{i}{p^2}\binom{p}{i}\binom{p}{j}p^{\frac{p-1}{p}+\frac{i}{p^2}+\frac{j}{p^3}}(\phi^{-2}(\delta(f)))^{p-2}(\phi^{-3}(\delta(f)))^{i-1}\\
&\quad \cdot (\phi^{-4}(\delta(f)))^{j}(\phi^{-2}(f))^{p-i}(\phi^{-3}(f))^{p-j} \biggr) \\
& \quad +\sum_{i=1}^{p-1}\sum_{j=1}^{p-1}p^{\frac{p-1}{p}+\frac{i}{p^2}+\frac{j}{p^3}+\frac{1}{p^4}}(\phi^{-2}(f))^{p-i}(\phi^{-3}(f))^{p-j}\gamma'_{i,j}\\
&\quad +\sum_{i=1}^{p-1}p^{\frac{p-1}{p}+\frac{i+1}{p^2}}(\phi^{-2}(f))^{p-i}\gamma'_{i,p}\\
&\pmod{(p,x_1,\dots,x_n)S^+}.
\end{align*}
First, suppose that $a\ge 3$. Then 
\begin{align*}
	f^{\frac{p^2-a}{p^2}} &= f^{\frac{p-1}{p}}f^{\frac{p-a}{p^2}} \\
&\equiv p^{\frac{p-1}{p}+\frac{p-a+1}{p^2}}\left(\sum_{i=1}^{p-a+1}\frac{1}{p}\binom{p}{i}\binom{p-a}{i-1}\right)\phi^{-2}((\delta(f))^{p-2}f^{p-1})(\phi^{-3}(\delta(f)))^{p-a+1} \\
&\pmod{(p^{\frac{p-1}{p}+\frac{p-a+1}{p^2}+\frac{1}{p^3}},x_1,\dots,x_n)S^+}.
\end{align*}
Consider the equation $(x+y)^{p}(x+y)^{p-a}=(x+y)^{2p-a}$ in $\Z[x,y]$. Looking at the coefficient of $x^{p-1}y^{p-a+1}$, we obtain
\[
	\sum_{i=1}^{p-a+1}\binom{p}{i}\binom{p-a}{i-1}=\binom{2p-a}{p-1} \in p\Z\setminus p^2\Z.
\]
Hence, there exists $C_2\in \Z\setminus p\Z$ such that
\begin{align*}
	f^{\frac{p^2-a}{p^2}} &\equiv p^{\frac{p-1}{p}+\frac{p-a+1}{p^2}}C_2x_1^{\frac{an}{p^2}}\cdots x_n^{\frac{an}{p^2}}(\phi^{-3}(\delta(f)))^{p-a+1}\\
&\pmod{(p^{\frac{p-1}{p}+\frac{p-a+1}{p^2}+\frac{1}{p^3}},x_1,\dots,x_n)S^+}.
\end{align*}
Since $p-a+1\le p-2$, $p(p-a+1)< p^2-p-1=an$. Let $Q$, $r$ be non-negative integers such that $p(p-a+1)=nQ+r$ and $0\le r<n$. Then $Q<a$. Note that $(ap+Q+1)n \le (p+1)an=(p+1)(p^2-p-1)<p^3$. Hence, there exists $C_3\in \Z\setminus p\Z$ such that
\begin{align*}
	f^{\frac{p^2-a}{p^2}} &\equiv p^{\frac{p-1}{p}+\frac{p-a+1}{p^2}}C_3x_1^{\frac{(ap+Q+1)n}{p^3}}\cdots x_r^{\frac{(ap+Q+1)}{p^3}}x_{r+1}^{\frac{(ap+Q)n}{p^3}}\cdots x_n^{\frac{(ap+Q)n}{p^3}} \\
&\pmod{(p^{\frac{p-1}{p}+\frac{p-a+1}{p^2}+\frac{1}{p^3}},x_1^{\frac{(ap+Q+2)n}{p^3}},\dots,x_r^{\frac{(ap+Q+2)n}{p^3}},x_{r+1}^{\frac{(ap+Q+1)n}{p^3}},\dots,x_n^{\frac{(ap+Q+1)n}{p^3}},\\
& \qquad \qquad x_1,\dots,x_r)S^+}.
\end{align*}
Hence, we have
\[
	f^{\frac{p^2-a}{p^2}}\notin (p,x_1,\dots,x_n)B
\]
for any perfectoid BCM $S^+$-algebra $B$. Therefore, $f^{(n-1)/n}\notin (p,x_1,\dots,x_n)B.$

 Next, suppose that $a\le 2$. Since $p^2-p-1$ is odd, it follows that $a=1$. If $p=3$, then there exists $C_4\in \Z\setminus p\Z$ such that
\begin{align*}
	f^{\frac{7}{9}} &\equiv 3^{\frac{2}{3}+\frac{2}{9}}C_4x_1^{\frac{5}{9}}\dots x_5^{\frac{5}{9}}\phi^{-3}(\delta(f))^2 \pmod{(3^{\frac{2}{3}+\frac{2}{9}+\frac{1}{27}},x_1,\dots,x_5)S^+}
\end{align*}
by a similar argument.
Hence, 
\begin{align*}
	f^{\frac{22}{27}} &\equiv 3^{\frac{2}{3}+\frac{2}{9}}C_4x_1^{\frac{5}{9}}\cdots x_5^{\frac{5}{9}}\phi^{-3}(\delta(f))^2\phi^{-3}(f) \\
	&\equiv 3^{\frac{2}{3}+\frac{2}{9}}C_4x_1^{\frac{5}{9}}\cdots x_5^{\frac{5}{9}}\cdot (-x_1^{\frac{10}{27}}x_2^{\frac{10}{27}}x_3^{\frac{5}{27}}x_4^{\frac{5}{27}}x_5^{\frac{5}{27}}) \\
&=-3^{\frac{2}{3}+\frac{2}{9}}C_4x_1^{\frac{25}{27}}x_2^{\frac{25}{27}}x_3^{\frac{20}{27}}x_4^{\frac{20}{27}}x_5^{\frac{20}{27}}\\
&\pmod{(3^{\frac{2}{3}+\frac{2}{9}+\frac{1}{27}},x_1,x_2, x_3^{\frac{25}{27}},x_4^{\frac{25}{27}},x_5^{\frac{25}{27}})S^+}.
\end{align*}
Hence, for any perfectoid BCM $S^+$-algebra $B$, we have $f^{22/27}\notin (3,x_1,\dots,x_5)B$.
Since $4/5<22/27$, we have $f^{4/5}\notin (3,x_1,\dots,x_5)B$.
Suppose that $p\ge 5$. Using Lemma \ref{refined p-root formula}, we can show that there exist $\gamma'_i\in S^+$ such that
\begin{align*}
 f^{\frac{1}{p^2}}&=\phi^{-2}(f)+p^{\frac{1}{p^2}}\phi^{-3}(\delta(f))-\sum_{i=1}^{p-1}\frac{1}{p}\binom{p}{i}p^{\frac{1}{p^2}+\frac{i}{p^3}}(\phi^{-4}(\delta(f)))^{i}(\phi^{-3}(f))^{p-i}\\
&+\sum_{i=1}^{p-1}p^{\frac{1}{p^2}+\frac{i}{p^3}+\frac{1}{p^4}}(\phi^{-3}(f))^{p-i}\gamma'_i+p^{\frac{2}{p^2}}\gamma'_p,
\end{align*}
and we have
\begin{align*}
f^{\frac{1}{p^3}} &\equiv \phi^{-3}(f)+p^{\frac{1}{p^3}}\phi^{-4}(\delta(f)) \pmod{(p^{\frac{1}{p^3}+\frac{1}{p^4}},x_1,\dots,x_n)S^+}.
\end{align*}
Since $(p+2)n=p^3+p^2-3p-2\ge p^3$ if $p\ge 5$, we have
\begin{align*}
f^{\frac{p^2-2}{p^2}} &= f^{\frac{p-1}{p}}f^{\frac{p-2}{p^2}} \\
& \equiv -\sum_{i=1}^{p-1}\sum_{j=1}^{p-1}\biggl(\frac{p-1}{p^2}\binom{p}{i}\binom{p}{j}\binom{p-2}{i-1}p^{\frac{p-1}{p}+\frac{p-1}{p^2}+\frac{j}{p^3}}\\
& \quad \cdot(\phi^{-2}(\delta(f)))^{p-2}(\phi^{-3}(\delta(f)))^{p-2}(\phi^{-4}(\delta(f)))^{j}(\phi^{-2}(f))^{p-1}(\phi^{-3}(f))^{p-j}\biggr)\\
& \quad +\sum_{i=1}^{p-1}\sum_{j=1}^{p-1}p^{\frac{p-1}{p}+\frac{p-1}{p^2}+\frac{j}{p^3}+\frac{1}{p^4}}(\phi^{-2}(f))^{p-1}(\phi^{-3}(f))^{p-j}\gamma''_{i,j} \\
&\pmod{(p,x_1,\dots,x_n)S^+}
\end{align*}
for some $\gamma''_{i,j}\in S^+$. Hence, 
\begin{align*}
	f^{\frac{p^3-p-2}{p^3}} &= f^{\frac{p^2-2}{p^2}}f^{\frac{p-2}{p^3}}\\
 &\equiv -\sum_{i=1}^{p-1}\sum_{j=1}^{p-1}\left(\frac{p-1}{p^2}\binom{p}{i}\binom{p}{j}\binom{p-2}{i-1}\binom{p-2}{j-1}\right)p^{\frac{p-1}{p}+\frac{p-1}{p^2}+\frac{p-1}{p^3}}\\
& \quad \cdot (\phi^{-2}(\delta(f)))^{p-2}(\phi^{-3}(\delta(f)))^{p-2}(\phi^{-4}(\delta(f)))^{p-1}(\phi^{-2}(f))^{p-1}(\phi^{-3}(f))^{p-1} \\
&\pmod{(p^{\frac{p-1}{p}+\frac{p-1}{p^2}+\frac{p-1}{p^3}+\frac{1}{p^4}},x_1,\dots,x_n)S^+}.
\end{align*}
By a similar argument as above,
\begin{align*}
\frac{p-1}{p^2}\left( \sum_{i=1}^{p-1}\binom{p}{i}\binom{p-2}{i-1}\right)\left(\sum_{j=1}^{p-1}\binom{p}{j}\binom{p-2}{j-1} \right) &=\frac{p-1}{p^2}\left(\binom{2p-2}{p-1}\right)^2 \in \Z\setminus p\Z.
\end{align*}
Since we have
\[
	f^{\frac{1}{p^4}} \equiv \phi^{-4}(f) \pmod{(p^{\frac{1}{p^4}},x_1,\dots,x_n)S^+}
\]
and
\[
(p^2+p+2)n=p^4-3p-2<p^4< p^4+p^2-4p-3=(p^2+p+3)n,
\]
we obtain
\begin{align*}
	f^{\frac{p^4-p^2-p-1}{p^4}} & =f^{\frac{p^3-p-2}{p^3}}f^{\frac{p-1}{p^4}} \\
	& \equiv p^{\frac{p-1}{p}+\frac{p-1}{p^2}+\frac{p-1}{p^3}}C_5x_1^{\frac{(p+1)n}{p^3}}\cdots x_n^{\frac{(p+1)n}{p^3}}(\phi^{-4}(\delta(f)))^{p-1}\phi^{-4}(f)^{p-1}\\
& \equiv p^{\frac{p-1}{p}+\frac{p-1}{p^2}+\frac{p-1}{p^3}}C_5x_1^{\frac{(p+1)n}{p^3}}\cdots x_n^{\frac{(p+1)n}{p^3}}\cdot C_6x_1^{\frac{2n}{p^4}}\cdots x_{p}^{\frac{2n}{p^4}}x_{p+1}^{\frac{n}{p^4}}\cdots x_n^{\frac{n}{p^4}} \\
&\pmod{(p^{\frac{p-1}{p}+\frac{p-1}{p^2}+\frac{p-1}{p^3}+\frac{1}{p^4}},x_1,\dots,x_{p},x_{p+1}^{\frac{(p^2+p+2)n}{p^4}},\dots, x_n^{\frac{(p^2+p+2)n}{p^4}})S^+}
\end{align*}
for some $ C_5, C_6\in \Z\setminus p\Z$. Hence, for any perfectoid BCM $S^+$-algebra $B$, we have
\[
	f^{\frac{p^4-p^2-p-1}{p^4}}\notin (p,x_1,\dots,x_n)B.
\]
Since 
\[
	\frac{p^4-p^2-p-1}{p^4}>\frac{n-1}{n},
\]
we obtain
\[
	f^{\frac{n-1}{n}}\notin (p,x_1,\dots,x_n)B.
\]
\end{proof}

\subsection{Negative results}
This subsection concerns Fermat-type hypersurfaces that are not BCM-regular.
\begin{prop} \label{Fermat type >=p^2}
	Let the notation be as in Setting \ref{setting Fermat type}. If $d\ge p^2$, then $R$ is not $+$-regular.
\end{prop}
\begin{proof}
	Let $f:=x_1^d+\dots+x_n^d$.
	It is enough to show that $\Z_p\llbracket x_0,\dots,x_n\rrbracket/(x_0^{p}+f)$ is not $+$-regular.
		By Proposition \ref{remark p-root formula},  for $1\le j \le p$, there exist $\gamma_j\in S^+$ such that 
	\[
	f^{\frac{1}{p}}=\phi^{-1}(f)+p^{\frac{1}{p}}\phi^{-2}(\delta(f))+\sum_{j=1}^{p}p^{\frac{1}{p}+\frac{j}{p^2}}(\phi^{-2}(f))^{p-j}\gamma_j
	\]
	in $S^+$. Hence, it follows that
\[
f^{\frac{1}{p}}\equiv p^{\frac{2}{p}}\gamma_p \pmod{(p,x_1,\dots,x_n)S^+}.
\]
	Therefore, $f^{(p-1)/p} \in (p,x_1,\dots,x_n)S^+,$ and $\Z_p\llbracket x_0,\dots,x_n\rrbracket/(x_0^{p}+f)$ is not $+$-regular.
\end{proof}

\begin{prop} \label{Fermat type not p^2-p-2}
	Let the notation be as in Setting \ref{setting Fermat type}. Let $a$ be a positive integer such that $1\le a \le p-1$. Suppose that 
\[
	\lceil \frac{p^2}{a+1}\rceil \le d \le n \le \lfloor \frac{p^2-p-1}{a}\rfloor-1.
\]
Then $R$ is not $+$-regular.
\end{prop}
\begin{proof}
	Suppose that $p$, $d$ and $n$ satisfy the above condition.
Let $f:=x_1^d+\dots+x_n^d$. By Proposition \ref{remark p-root formula}, there exist $\gamma_j \in S^+$ such that
\[
	f^{\frac{1}{p}}=\phi^{-1}(f)+p^{\frac{1}{p}}\phi^{-2}(\delta(f))+\sum_{j=1}^{p}p^{\frac{1}{p}+\frac{j}{p^2}}(\phi^{-2}(f))^{p-j}\gamma_j
\]
in $S^+$. Since $\phi^{-1}(f) \in (x_1,\dots,x_n)S^+$, we have
\begin{align*}
	f^{\frac{p-1}{p}} &\equiv p^{\frac{p-1}{p}}(\phi^{-2}(\delta(f)))^{p-1}-\sum_{j=1}^{p-1}p^{\frac{p-1}{p}+\frac{j}{p^2}}(\phi^{-2}(\delta(f)))^{p-2}(\phi^{-2}(f))^{p-j}\gamma_j \\
&\pmod {(p,x_1,\dots,x_n)S^+}.
\end{align*}
Since $p(p-1)\ge an+1$ and $d(a+1)\ge p^2$, we have $\delta(f)^{p-1}\in (p,x_1^{p^2},\dots,x_n^{p^2})S$. Hence, we obtain
\begin{align*}
	f^{\frac{p-1}{p}} &\equiv -\sum_{j=1}^{p-1}p^{\frac{p-1}{p}+\frac{j}{p^2}}(\phi^{-2}(\delta(f)))^{p-2}(\phi^{-2}(f))^{p-j}\gamma_j \\
&\pmod {(p,x_1,\dots,x_n)S^+}.
\end{align*}
By Theorem \ref{p^e-root formula}, 
\begin{align*}
	f^{\frac{1}{p^2}} &\equiv \phi^{-2}(f) \pmod{p^{\frac{1}{p^2}}S^+}.
\end{align*}
Since $an+1\le p^2-p-a$, we have $\delta(f)^{p-2}f^{p-a} \in(x_1^{p^2},\dots,x_n^{p^2})$.
Hence,
\begin{align*}
	f^{\frac{p^2-a-1}{p^2}} &= f^{\frac{p-1}{p}} f^{\frac{p-a-1}{p^2}}\in (p,x_1,\dots,x_n)S^+.
\end{align*}
Since 
\[
	\frac{p^2-a-1}{p^2} \le\frac{d-1}{d},
\]
we have
\[
	f^{\frac{d-1}{d}}\in (p,x_1,\dots,x_n)S^+,
\]
which implies that $R$ is not $+$-regular.
\end{proof}

\subsection{Main Theorem}
Summarizing the above propositions, we obtain the main theorem.
\begin{thm} \label{thm Fermat type}
	Let the notation be as in Setting \ref{setting Fermat type}.
\begin{enumerate}
	\item $R$ is BCM-regular if one of the following conditions holds.
		\begin{enumerate}
			\item $d\le \min \{n,  p\}$.
			\item $d\le n$ and there exists a positive integer $a$ such that $1\le a \le p-1$ and one of the following holds:
			\begin{enumerate}
				\item $n=\lfloor (p^2-1)/a\rfloor.$
				\item $an=p^2-p$.
				\item $an=p^2-p-1$.
				\item $an=p^2-a-1$.
			\end{enumerate}
		\end{enumerate}
	\item $R$ is not $+$-regular if one of the following conditions holds.
		\begin{enumerate}
			\item $d>n$.
			\item $d\ge p^2$.
			\item There exists a positive integer such that $1\le a \le p-1$ and
			\[
				\lceil \frac{p^2}{a+1} \rceil \le d \le n \le \lfloor \frac{p^2-p-1}{a}\rfloor-1. 
			\]
		\end{enumerate}
\end{enumerate}
\end{thm}
\begin{proof}
	\begin{enumerate}
	\item	\begin{enumerate}
		\item This follows from Proposition \ref{Fermat type <=p}.
		\item \begin{enumerate}
			\item This follows from Proposition \ref{Fermat type p^2-1/a}.
			\item This follows from Proposition \ref{Fermat type an=p^2-p}.
			\item This also follows from Proposition \ref{Fermat type an=p^2-p}.
			\item This follows from Proposition \ref{Fermat type an=p^2-a-1}.
			\end{enumerate}
		\end{enumerate}
		\item
			\begin{enumerate}
			\item This follows from Corollary \ref{+-regular and canonical diagonal hypersurface}.
			\item This follows from Proposition \ref{Fermat type >=p^2}.
			\item This follows from Proposition \ref{Fermat type not p^2-p-2}.
			\end{enumerate}
	\end{enumerate}
\end{proof}

\begin{eg} \label{eg p=3}
With notation as in Setteing \ref{setting Fermat type}, suppose that $p=3$. $R$ is BCM-regular if and only if $d \le \min\{n, 8\}$.

Indeed, if $d \ge p^2=9$, then $R$ is not BCM-regular by (2)(a).	
Hence, we may assume that $d=n<9$. By (1)(a), $R$ is BCM-regular if $n\le 3$. By (1)(b)(i), $R$ is BCM-regular if $n=4,8$. By (1)(b)(ii)and (iii), $R$ is BCM-regular if $n=5,6$. By (1)(b)(iv), $R$ is BCM-regular if $n=7$.
\end{eg}
\begin{eg} \label{eg p=5}
With notation as in Setting \ref{setting Fermat type}, suppose that $p=5$ and $(d,n)\neq (21,21), (21,22),(22,22)$. Then $R$ is BCM-regular if and only if $d$ and $n$ satisfy one of the following conditions:
\begin{enumerate}
	\item $d\le \min\{12,n\}.$
	\item $d\le \min\{24,n\}$ and $19\le n$.
\end{enumerate}
\end{eg}
Since the case $(d,n)=(9,9)$ does not follow from the above theorem, we need the following lemma.
\begin{lem}
	With notation as in Setting \ref{setting Fermat type}, suppose that $p=5$, $d=n=9$. Then $R$ is BCM-regular.
\end{lem}
\begin{proof}
	We use \cite[Theorem 6.4]{Yoshikawa25}. We can check $(s_0,s_1,s_2,s_3)=(0,p-1,p-1,0)$ (for the definition of $s_i$, see Definition \ref{def of splitting-order sequences}). Let $\overline{R}:=R/pR$, $K:=\operatorname{Ker}(F:H_\m^d(\overline{R})\to H_\m^d(\overline{R}))$ and $t:=\min\{l\in \Z\mid K_l\neq 0\}.$ It is enough to show the following claim.
\begin{cl}
	$t>-125$.
\end{cl}
\begin{clproof}
	Let $S:=\mathbb{F}_5[x_0,\dots,x_9]/(x_0^9+\dots+x_9^9)$. Then $K=\operatorname{Ker}(F:H_\m^d(S)\to H_\m^d(S))$.
	Since $S$ is Gorenstein and $\omega_S\cong S(-1)$, we have an exact sequence
	\[
		 F_*(S(-1)) \to S(-1) \to K^{\vee}\to 0,
	\]
	where $K^{\vee}$ is the Matlis dual of $K$. Let $I:=\operatorname{Im}(\operatorname{Tr}_F:F_*S\to S)$. Then $\tau(S)\subseteq I$, where $\tau(S)$ denotes the test ideal of $S$. By \cite[Example 3.12]{Huneke98}, $(x_0^8,\dots,x_9^8)\subseteq \tau(S)$. Hence, $S_{\ge 71}\subseteq \tau(S)\subseteq I$, and we have
\[
	(K^{\vee})_{\ge 72} \cong [(S/I)(-1)]_{\ge 72} =0.
\]
Therefore, $t> -72$. 
\end{clproof}
\end{proof}

\begin{ques}
Is $R$ BCM-regular when $(d,n)=(21,21),(21,22),(22,22)$ ?
\end{ques}

\section{Plus-pure thresholds of hypersurfaces}
In this section, we use our formula to estimate plus-pure thresholds in several cases.

\subsection{Plus-pure thresholds of diagonal hypersurfaces}
First, we establish lower bounds for plus-pure thresholds of diagonal hypersurfaces in terms of $p$-adic expansions.

The following version of the $p$-th root formula will be needed in the proof of the theorem.
\begin{lem} \label{p-th root formula for ppt}
	Let $p$ be a prime number, $n\ge 1$ be an integer, $S:=\Z_p\llbracket x_1,\dots,x_n\rrbracket$ and $f\in S$. For any $e'\ge1$ and $e\ge 1$, there exist $\beta_{e,e'}, \gamma_{e,e'}\in S^+$ such that
	\begin{align*}
		f^{\frac{1}{p^{e'}}} & \equiv \phi^{-e'}(f)+\sum_{i=1}^{e} (-1)^{i-1}p^{\frac{1}{p^{e'}}+\dots+\frac{1}{p^{e'+i-1}}} (\phi^{-(e'+i-1)}(f))^{p^{i-1}-1}\phi^{-(e'+i)}(\delta(f)) \\
		&+p^{\frac{1}{p^{e'}}+\dots+\frac{1}{p^{e'+e}}}(\phi^{-(e'+e)}(f))^{p^{e'+e-1}-1}\beta_{e,e'} +p^{\frac{1}{p^{e'}}+\dots+\frac{1}{p^{e'+e-1}}+\frac{2}{p^{e'+e}}}\gamma_{e,e'}.
	\end{align*}
\end{lem}
\begin{proof}
	As in the proof of Theorem \ref{p^e-root formula}, it is enough to show the case where $e'=1$. We show the theorem by induction on $e$. If $e=1$, then there exist $\beta_1, \gamma_1\in S^+$ such that
	\begin{align*}
		f^{\frac{1}{p}}=\phi^{-1}(f)+p^{\frac{1}{p}}\phi^{-2}(\delta(f))+p^{\frac{1}{p}+\frac{1}{p^2}}(\phi^{-2}(f))^{p-1}\beta_1+p^{\frac{1}{p}+\frac{2}{p^2}}\gamma_1
	\end{align*}
	by Lemma \ref{refined p-root formula} and a similar argument when $p=2$. Suppose that there exist $\beta_e, \gamma_e\in S^+$ such that 
	\begin{align} \label{align 6.1.4}
		f^{\frac{1}{p}} &= \phi^{-1}(f)+\sum_{i=1}^{e}(-1)^{i-1}p^{\frac{1}{p}+\dots+\frac{1}{p^i}}(\phi^{-i}(f))^{p^{i-1}-1}\phi^{-(i+1)}(\delta(f))\\
		&\quad +p^{\frac{1}{p}+\dots+\frac{1}{p^{e+1}}}(\phi^{-(e+1)}(f))^{p^{e}-1}\beta_e+p^{\frac{1}{p}+\dots+\frac{1}{p^e}+\frac{2}{p^{e+1}}}\gamma_e. \notag
	\end{align}
	By the proof of Lemma \ref{refined p-root formula}, 
	\begin{align} \label{align 6.1.5}
		(f^{\frac{1}{p}}-\phi^{-1}(f))^{p}\equiv p\phi^{-1}(\delta(f))-p(f^{\frac{1}{p}}-\phi^{-1}(f))(\phi^{-1}(f))^{p-1} \pmod{p^{1+\frac{2}{p}}S^+}.
	\end{align}
	Substituting \ref{align 6.1.4} into \ref{align 6.1.5}, we have
	\begin{align*}
		(f^{\frac{1}{p}}-\phi^{-1}(f))^p &\equiv p\phi^{-1}(\delta(f))+\sum_{i=1}^{e}(-1)^ip^{1+\frac{1}{p}+\dots+\frac{1}{p^i}}(\phi^{-i}(f))^{p^i-1}\phi^{-(i+1)}(\delta(f))\\
		&\quad -p^{1+\frac{1}{p}+\dots+\frac{1}{p^{e+1}}}(\phi^{-(e+1)}(f))^{p^{e+1}-1}\beta_e \pmod{p^{1+\frac{1}{p}+\dots+\frac{1}{p^e}+\frac{2}{p^{e+1}}}S^+}.
	\end{align*}
	Therefore, there exist $\beta_{e+1}, \gamma_{e+1}\in S^+$ such that
	\begin{align*}
		f^{\frac{1}{p}} &= \phi^{-1}(f)+\sum_{i=1}^{e+1}(-1)^{i-1}p^{\frac{1}{p}+\dots+\frac{1}{p^{i+1}}}(\phi^{-i}(f))^{p^{i-1}-1}\phi^{-(i+1)}(\delta(f))\\
		& \quad +p^{\frac{1}{p}+\dots+\frac{1}{p^{e+2}}}(\phi^{-(e+2)})^{p^{e+1}-1}\beta_{e+1}+p^{\frac{1}{p}+\dots+\frac{1}{p^{e+1}}+\frac{2}{p^{e+2}}}\gamma_{e+1}.
	\end{align*}
\end{proof}

In the proof, we use the notation from \cite[Section 2]{Hernandez15}.

\begin{notation} \label{notation p-adic expansion}
For $\gamma\in (0,1]\subseteq \R$, let $\gamma^{(i)}$ be the unique integer $\gamma^{(i)}\in \{0,1,\dots, p-1\}$ such that
\[
	\alpha=\sum_{i= 1}^{\infty}\frac{\gamma^{(i)}}{p^i}
\] 
and $\gamma^{(i)}$ is not eventually zero. We adopt the convention that $0^{(e)}=0$ for $e\ge 1$. 
For $e\ge 1$, we use $\langle \gamma \rangle_{e}$ to denote 
\[
\sum_{i=1}^e\frac{\gamma^{(i)}}{p^i}.
\]
\end{notation}

\begin{thm} \label{diagonal hypersurface brt}
	Let $p$ be a prime number, $n, d_1,\dots,d_n\ge 2$ be positive integers, $S:=\Z_p\llbracket x_1,\dots,x_n \rrbracket$ and $f:=x_1^{d_1}+\dots+x_n^{d_n}$. Suppose that there exist non-negative rational numbers $\alpha, \alpha_1,\dots, \alpha_n$ such that $\alpha=\alpha_1+\dots+\alpha_n\le 1$ and $d_i\alpha_i \le 1$ for $i=1,\dots,n$. Moreover, assume that there is at most one carry at each digit when adding the $p$-adic expansions of $\alpha_1, \dots, \alpha_n$ as in Notation \ref{notation p-adic expansion}. Then $\ppt(f)\ge \alpha$.
\end{thm}
\begin{proof}
We may assume that $\alpha>0$. Moreover, since we can omit the variable $x_i$ whenever $\alpha_i=0$, we may assume that $\alpha_i>0$ for all $i$.
For $i \ge 1$, let $s_i$ be the $i$-th member of the set of positive integers $e$ such that $\alpha^{(e)}>0$. We define $t_i$ as follows:
\begin{enumerate}
	\item If $\alpha_{1}^{(s_i)}+\dots+\alpha_{n}^{(s_i)} \equiv \alpha^{(s_i)} \pmod{p}$, then $t_i:=s_i$.
	\item If $\alpha_{1}^{(s_i)}+\dots+\alpha_{n}^{(s_i)} \not\equiv \alpha^{(s_i)} \pmod{p}$, then
\[
	t_i:=\min\{e \in \N\mid \text{$e> s_i$ and $\alpha_{1}^{(e)}+\dots + \alpha_{n}^{(e)}\ge p$} \}.
\]
\end{enumerate}
Note that in case (2), since a carry must occur, $t_i$ is well defined and $t_i\le s_{i+1}$. By Lemma \ref{p-th root formula for ppt}, for any $m\ge 1$, there exist $\beta_m, \gamma_m\in S^+$ such that
\begin{align*}
	f^{\frac{1}{p^{s_m}}}= &\phi^{-s_m}(f)+\sum_{i=1}^{t_m-s_m} (-1)^{i-1}p^{\frac{1}{p^{s_m}}+\dots+\frac{1}{p^{s_m+i-1}}} (\phi^{-(i+s_m-1)}(f))^{p^{i-1}-1}\phi^{-(i+s_m)}(\delta(f))\\
&+p^{\frac{1}{p^{s_m}}+\dots+\frac{1}{p^{t_m}}}(\phi^{-t_m}(f))^{p^{t_m-s_m}-1}\beta_{m}+p^{\frac{1}{p^{s_m}}+\dots+ \frac{1}{p^{t_m-1}}+\frac{2}{p^{t_m}}}\gamma_m.
\end{align*}

\begin{cl}
For any $m\ge 1$, there exists an integer $C_m$ coprime to $p$ such that
\begin{enumerate}
	\item If $t_m<s_{m+1}$, then
	\begin{align*}
		f^{\langle \alpha \rangle_{s_m}} &= f^{\frac{\alpha^{(s_1)}}{p^{s_1}}}\cdots f^{\frac{\alpha^{(s_m)}}{p^{s_m}}} \\
& \equiv
C_m p^{\sum_{i=1}^{m}\frac{1}{p^{s_i}}+\dots+\frac{1}{p^{t_i-1}}}x_1^{d_1\langle \alpha_1\rangle_{t_m}}\cdots x_n^{d_n\langle \alpha_n\rangle_{t_m}} \\
	&\pmod{(p^{\sum_{i=1}^{m}(\frac{1}{p^{s_i}}+\dots+\frac{1}{p^{t_i-1}})+\frac{1}{p^{t_m}}},x_1^{d_1\alpha_1},\dots, x_n^{d_n\alpha_n})S^+}.
	\end{align*}
	\item If $t_m=s_{m+1}$, then
	\begin{align*}
		f^{\langle \alpha \rangle_{s_m}} &= f^{\frac{\alpha^{(s_1)}}{p^{s_1}}}\cdots f^{\frac{\alpha^{(s_m)}}{p^{s_m}}}\\
& \equiv
C_m p^{\sum_{i=1}^{m}\frac{1}{p^{s_i}}+\dots+\frac{1}{p^{t_i-1}}}x_1^{d_1\langle \alpha_1\rangle_{t_m-1}}\cdots x_n^{d_n\langle \alpha_n \rangle_{t_m-1}}\phi^{-t_m}(\delta(f)) \\
&\pmod{(p^{\sum_{i=1}^{m}(\frac{1}{p^{s_i}}+\dots+\frac{1}{p^{t_i-1}})+\frac{1}{p^{t_m}}} x_1^{d_1\langle \alpha_1 \rangle_{t_m-1}}\cdots x_n^{d_n\langle \alpha_n\rangle_{t_m-1}}(\phi^{-t_{m}}(f))^{p-1},\\
&\qquad  p^{\sum_{i=1}^{m}(\frac{1}{p^{s_i}}+\dots+\frac{1}{p^{t_i-1}})+\frac{2}{p^{t_{m}}}}, x_1^{d_1\alpha_1},\dots, x_n^{d_n\alpha_n})S^+}.
\end{align*}
\end{enumerate}
\end{cl}
\begin{clproof}
	We show the claim by induction on $m$.
	First, suppose that $m=1$. 
If $t_1=s_1$, then $\alpha^{(s_1)}=\alpha_1^{(s_1)}+\cdots + \alpha_n^{(s_1)}$. Hence,
	\[
		f^{\langle \alpha \rangle_{s_1}}=f^{\frac{\alpha^{(s_1)}}{p^{s_1}}}\equiv \frac{\alpha^{(s_1)}!}{\alpha_1^{(s_1)}!\dots \alpha_n^{(s_1)}!}x_1^{\frac{d_1 \alpha_{1}^{(s_1)}}{p^{s_1}}}\cdots x_n^{\frac{d_n \alpha_{n}^{(s_1)}}{p^{s_1}}} \pmod{(p^{\frac{1}{p^{s_1}}},x_1^{d_1\alpha_1},\dots, x_n^{d_n\alpha_n})S^+}.
	\]
	If $s_1<t_1<s_2$, then $\langle \alpha \rangle_{s_1}=\langle \alpha_1 \rangle_{t_1}+\cdots +\langle \alpha_n \rangle_{t_1}$. Hence, there exists an integer $C$ coprime to $p$ such that 
	\begin{align*}
		&f^{\langle \alpha \rangle_{s_1}} \\
&\equiv \alpha^{(s_1)}(-1)^{t_1-s_1-1}p^{\frac{1}{p^{s^1}}+\dots+\frac{1}{p^{t_1-1}}}\\
& \quad \cdot (\phi^{-s_1}(f))^{\alpha^{(s_1)}-1}(\phi^{-(s_1+1)}(f))^{p-1}\dots (\phi^{-(t_1-1)}(f))^{p-1}\phi^{-t_1}(\delta(f)) \\
&\equiv Cp^{\frac{1}{p^{s^1}}+\dots+\frac{1}{p^{t_1-1}}}x_1^{d_1 \langle \alpha \rangle_{t_1}}\cdots x_n^{d_n\langle \alpha_n \rangle_{t_1}} \pmod{(p^{\frac{1}{p^{s^1}}+\dots+\frac{1}{p^{t_1}}}, x_1^{d_1\alpha_1},\dots, x_n^{d_n\alpha_n})S^+}.
\end{align*}
If $t_1=s_2$, then $\langle \alpha \rangle_{s_1}=\langle \alpha_1\rangle_{t_1-1}+\cdots+\langle \alpha_n \rangle_{t_1-1}+p/p^{t_1}$.
Hence, there exists an integer $C$ coprime to $p$ such that 
	\begin{align*}
		&f^{\langle \alpha \rangle_{s_1}} \\
&\equiv \alpha^{(s_1)}(-1)^{t_1-s_1-1}p^{\frac{1}{p^{s^1}}+\dots+\frac{1}{p^{t_1-1}}}\\
& \quad \cdot (\phi^{-s_1}(f))^{\alpha^{(s_1)}-1}(\phi^{-(s_1+1)}(f))^{p-1}\dots (\phi^{-(t_1-1)}(f))^{p-1}\phi^{-t_1}(\delta(f)) \\
&\equiv Cp^{\frac{1}{p^{s^1}}+\dots+\frac{1}{p^{t_1-1}}}x_1^{d_1 \langle \alpha_1\rangle_{t_1-1}}\cdots x_n^{d_n\langle \alpha_n \rangle_{t_1-1}}\phi^{-t_1}(\delta(f)) \\
&\pmod{(p^{\frac{1}{p^{s^1}}+\dots+\frac{1}{p^{t_1}}}x_1^{d_1\langle \alpha_1\rangle_{t_1-1}}\cdots x_n^{d_n\langle \alpha_n \rangle_{t_1-1}}(\phi^{-t_{1}}(f))^{p-1}, p^{\frac{1}{p^{s^1}}+\dots+\frac{1}{p^{t_1-1}}+\frac{2}{p^{t_1}}}, \\
& \qquad \qquad x_1^{d_1\alpha_1},\dots, x_n^{d_n\alpha_n})S^+},
\end{align*}
which shows the case where $m=1$. Next, suppose that the claim holds for $m$. If $t_m<s_{m+1}$ and $t_{m+1}<s_{m+2}$, then there exists an integer $C$ coprime to $p$ such that
\begin{align*}
	f^{\frac{\alpha^{(s_{m+1})}}{p^{s_{m+1}}}} & \equiv Cp^{\frac{1}{p^{s_{m+1}}}+\dots+\frac{1}{p^{t_{m+1}-1}}}x_1^{d_1\left(\frac{\alpha_1^{(s_{m+1})}}{p^{s_{m+1}}}+\dots +\frac{\alpha_1^{(t_{m+1})}}{p^{t_{m+1}}}\right)}\cdots x_n^{d_n\left(\frac{\alpha_n^{(s_{m+1})}}{p^{s_{m+1}}}+\dots +\frac{\alpha_{n}^{(t_{m+1})}}{p^{t_{m+1}}}\right)} \\
&\pmod{(p^{\frac{1}{p^{s_{m+1}}}+\dots+\frac{1}{p^{t_{m+1}}}},x_1^{d_1\alpha_1-d_1\langle \alpha_1 \rangle_{t_m}}, \dots, x_n^{d_n\alpha_n-d_n\langle \alpha_n \rangle_{t_m}})S^+}.
\end{align*}
Hence, there exists an integer $C'$ coprime to $p$ such that
\begin{align*}
	f^{\langle \alpha \rangle_{s_{m+1}}} &= f^{\frac{\alpha^{(s_1)}}{p^{s_1}}}\cdots f^{\frac{\alpha^{(s_{m+1})}}{p^{s_{m+1}}}} \\
& \equiv C' p^{\sum_{i=1}^{m+1}\frac{1}{p^{s_i}}+\dots+\frac{1}{p^{t_i-1}}}x_1^{d_1\langle \alpha_1 \rangle_{t_{m+1}}}\cdots x_{n}^{d_n\langle \alpha_n\rangle_{t_{m+1}}} \\
&\pmod{(p^{\sum_{i=1}^{m+1}(\frac{1}{p^{s_i}}+\dots+\frac{1}{p^{t_i-1}})+\frac{1}{p^{t_{m+1}}}},x_1^{d_1\alpha_1},\dots, x_n^{d_n\alpha_n})S^+}.
\end{align*}
If $t_m<s_{m+1}$ and $t_{m+1}=s_{m+2}$, then there exists an integer coprime to $p$ such that
\begin{align*}
	&f^{\frac{\alpha^{(s_{m+1})}}{p^{s_{m+1}}}} \\
    &\equiv Cp^{\frac{1}{p^{s_{m+1}}}+\dots+\frac{1}{p^{t_{m+1}-1}}}\\
    & \quad \cdot x_1^{d_1\left(\frac{\alpha_1^{(s_{m+1})}}{p^{s_{m+1}}}+\dots +\frac{\alpha_1^{(t_{m+1}-1)}}{p^{t_{m+1}-1}}\right)}\cdots x_n^{d_n\left(\frac{\alpha_n^{(s_{m+1})}}{p^{s_{m+1}}}+\dots +\frac{\alpha_{n}^{(t_{m+1}-1)}}{p^{t_{m+1}-1}}\right)}\phi^{-t_{m+1}}(\delta(f)) \\
&\pmod{(p^{\frac{1}{p^{s_{m+1}}}+\dots+\frac{1}{p^{t_{m+1}}}} x_1^{d_1\left(\frac{\alpha_1^{(s_{m+1})}}{p^{s_{m+1}}}+\dots +\frac{\alpha_1^{(t_{m+1}-1)}}{p^{t_{m+1}-1}}\right)}\\
& \qquad \cdots x_n^{d_n\left(\frac{\alpha_n^{(s_{m+1})}}{p^{s_{m+1}}}+\dots +\frac{\alpha_{n}^{(t_{m+1}-1)}}{p^{t_{m+1}-1}}\right)}(\phi^{-{t_{m+1}}}(f))^{p-1}, p^{\frac{1}{p^{s_{m+1}}}+\dots+\frac{1}{p^{t_{m+1}-1}}+\frac{2}{p^{t_{m+1}}}}, \\
& \qquad \qquad x_1^{d_1\alpha_1-d_1\langle \alpha_1 \rangle_{t_m}}, \dots, x_n^{d_n\alpha_n-d_n\langle \alpha_n \rangle_{t_m}})S^+},
\end{align*}
whence the claim for $m+1$ follows if $t_m<s_{m+1}$. Next, suppose that $t_{m}=s_{m+1}$.
Since $f^{\frac{\alpha^{(s_{m+1})}}{p^{s_{m+1}}}}\equiv (\phi^{-s_{m+1}}(f))^{\alpha^{(s_{m+1})}} \pmod{p^{\frac{1}{p^{s_{m+1}}}}S^+}$, we have
\begin{align*}
	&f^{\frac{\alpha^{(s_{m+1})}}{p^{s_{m+1}}}}(p^{\sum_{i=1}^{m}(\frac{1}{p^{s_i}}+\dots+\frac{1}{p^{t_i-1}})+\frac{1}{p^{t_m}}} x_1^{d_1\langle \alpha_1 \rangle_{t_m-1}}\cdots x_n^{d_n\langle \alpha_n\rangle_{t_m-1}}(\phi^{-t_{m}}(f))^{p-1},\\
	&\qquad  p^{\sum_{i=1}^{m}(\frac{1}{p^{s_i}}+\dots+\frac{1}{p^{t_i-1}})+\frac{2}{p^{t_{m}}}}, x_1^{d_1\alpha_1},\dots, x_n^{d_n\alpha_n})S^+\\
	& \subseteq (p^{\sum_{i=1}^{m+1}(\frac{1}{p^{s_i}}+\dots+\frac{1}{p^{t_i-1}})+\frac{2}{p^{t_{m+1}}}},x_1^{d_1\alpha_1},\dots, x_n^{d_n\alpha_n})S^+.
\end{align*}
Note that we have
\begin{align*}
	\alpha^{(s_{m+1})}=
	\begin{cases}
		\alpha_1^{(s_{m+1})}+\dots +\alpha_m^{(s_{m+1})}-p & \text{if $s_{m+1}=t_{m+1}$,}\\
		\alpha_1^{(s_{m+1})}+\dots +\alpha_m^{(s_{m+1})}-p+1 & \text{if $s_{m+1}<t_{m+1}$.}
	\end{cases}
\end{align*}
If $t_{m+1}<s_{m+2}$, then there exists an integer $C$ coprime to $p$ such that
\begin{align*}
	& \phi^{-t_m}(\delta(f))f^{\frac{\alpha^{(s_{m+1})}}{p^{s_{m+1}}}} \\
	& \equiv Cp^{\frac{1}{p^{s_{m+1}}}+\dots+\frac{1}{p^{t_{m+1}-1}}}x_1^{d_1\left(\frac{\alpha_1^{(s_{m+1})}}{p^{s_{m+1}}}+\dots +\frac{\alpha_{1}^{(t_{m+1})}}{p^{t_{m+1}}}\right)}\cdots x_{n}^{d_n\left(\frac{\alpha_n^{(s_{m+1})}}{p^{s_{m+1}}}+\dots +\frac{\alpha_{n}^{(t_{m+1})}}{p^{t_{m+1}}}\right)}\\
	& \pmod{(p^{\frac{1}{p^{s_{m+1}}}+\dots+\frac{1}{p^{t_{m+1}}}},x_1^{d_1\alpha_1-d_1\langle \alpha_1 \rangle_{t_m-1}}, \dots, x_n^{d_n\alpha_n-d_n\langle \alpha_n \rangle_{t_m-1}})S^+}.
\end{align*}
If $t_{m+1}=s_{m+2}$, there exists an integer $C$ coprime to $p$ such that
\begin{align*}
	& \phi^{-t_m}(\delta(f))f^{\frac{\alpha^{(s_{m+1})}}{p^{s_{m+1}}}} \\
	& \equiv Cp^{\frac{1}{p^{s_{m+1}}}+\dots+\frac{1}{p^{t_{m+1}-1}}}\\
    & \quad \cdot x_1^{d_1\left(\frac{\alpha_1^{(s_{m+1})}}{p^{s_{m+1}}}+\dots +\frac{\alpha_{1}^{(t_{m+1}-1)}}{p^{t_{m+1}-1}}\right)}\cdots x_{n}^{d_n\left(\frac{\alpha_n^{(s_{m+1})}}{p^{s_{m+1}}}+\dots +\frac{\alpha_{n}^{(t_{m+1}-1)}}{p^{t_{m+1}-1}}\right)}\phi^{-t_{m+1}}(\delta(f))\\
	& \pmod{(p^{\frac{1}{p^{s_{m+1}}}+\dots+\frac{1}{p^{t_{m+1}}}}\\
    &\qquad \qquad \cdot x_1^{d_1\left(\frac{\alpha_1^{(s_{m+1})}}{p^{s_{m+1}}}+\dots +\frac{\alpha_{1}^{(t_{m+1}-1)}}{p^{t_{m+1}-1}}\right)}\cdots x_{n}^{d_n\left(\frac{\alpha_n^{(s_{m+1})}}{p^{s_{m+1}}}+\dots +\frac{\alpha_{n}^{(t_{m+1}-1)}}{p^{t_{m+1}-1}}\right)}(\phi^{-t_{m+1}}(f))^{p-1},\\
		&\qquad \qquad p^{\frac{1}{p^{s_{m+1}}}+\dots+\frac{1}{p^{t_{m+1}-1}}+\frac{2}{p^{t_{m+1}}}}, x_1^{d_1\alpha_1-d_1\langle \alpha_1 \rangle_{t_m-1}}, \dots, x_n^{d_n\alpha_n-d_n\langle \alpha_n \rangle_{t_m-1}})S^+}.
\end{align*}

Hence, in either case, the claim holds for $m+1$.
\end{clproof}

By the above claim,
\[
f^{\langle \alpha \rangle_{s_m}} \notin (p,x_1,\dots,x_n)S^+
\]
for any $m\ge 1$. Since $\lim_{m\to \infty}\langle \alpha \rangle_{s_m}=\alpha$, we obtain
$\ppt(f)\ge \alpha$.

\end{proof}

\begin{eg}
	Let $S:=\Z_p\llbracket x,y,z \rrbracket$ and $f=x^2+y^3+z^6$. Then $\ppt(f)=1$. Indeed, let $c_i$ be the carry at the $i$-th fractional digit when adding the $p$-adic expansions of $1/2$, $1/3$ and $1/6$. By Theorem \ref{diagonal hypersurface brt}, it is enough to show that $c_i\le 1$ for any $i\ge 2$.
First, suppose that $p=2$. Then
\begin{align*}
	\frac{1}{2} &= 0.011111\dots_{(2)}, \\
	\frac{1}{3} &= 0.010101\dots_{(2)}, \\
	\frac{1}{6} &= 0.001010\dots_{(2)},
\end{align*}
where $(-)_{(p)}$ denotes the $p$-adic expansion. Hence, $c_i=1$ for any $i\ge 2$.

Next, suppose that $p=3$. Then
\begin{align*}
	\frac{1}{2} = 0.111\dots_{(3)}, \quad \frac{1}{3} = 0.022\dots_{(3)}, \quad \frac{1}{6} &= 0.011\dots_{(3)}.
\end{align*}
Hence, $c_i=1$ for any $i \ge 2$.

Thirdly, suppose that $p\equiv 1 \pmod{6}$. Let $n:=(p-1)/6$. Then
\begin{align*}
	\frac{1}{2} = \sum_{i=1}^\infty \frac{3n}{p^i},\quad \frac{1}{3} = \sum_{i=1}^\infty \frac{2n}{p^i}, \quad \frac{1}{6} = \sum_{i=1}^\infty \frac{n}{p^i}.
\end{align*}
Hence, $c_i=0$ for any $i\ge 2$.

Lastly, suppose that $p \equiv 5 \pmod{6}$. Let $n:=(p^2-1)/6$. Then
\begin{align*}
	\frac{1}{2}= \sum_{i=1}^{\infty}\frac{3n}{p^{2i}} \quad \frac{1}{3}=\sum_{i=1}^{\infty}\frac{2n}{p^{2i}}, \quad \frac{1}{6}=\sum_{i=1}^{\infty}\frac{n}{p^{2i}}.
\end{align*}
We see that
\[
	n \equiv \frac{5p-1}{6}, \quad 2n \equiv \frac{2p-1}{3}, \quad 3n \equiv \frac{p-1}{2} \pmod{p}.
\]
Since $6n<p^2$, $c_i=0$ if $i$ is odd.
Since
\[
	\frac{5p-1}{6}+\frac{2p-1}{3}+\frac{p-1}{2}=2p-1,
\]
$c_i=1$ if $i$ is even.
\end{eg}

The following corollary gives an affirmative answer to \cite[Question 4.4]{CPQGST25}.

\begin{cor} \label{binomial ppt}
Let $p$ be a prime number, $a, b\ge 2$ be positive integers, $S:=\Z_p\llbracket x,y \rrbracket$ and $f:=x^a+y^b$.
Then
\[
	\ppt(f)=\frac{1}{a}+\frac{1}{b}.
\]
\end{cor}
\begin{proof}
	Let $\alpha=1/a+1/b$, $\alpha_1=1/a$ and $\alpha_2=1/b$. 

First, we show that $\ppt(f)\ge \alpha.$ By Theorem \ref{diagonal hypersurface brt}, it is enough to show that there is at most one carry at each digit when adding the $p$-adic expansions of $\alpha_1$ and $\alpha_2$. This follows from the observation that $\alpha_1^{(e)}+\alpha_2^{(e)}\le 2p-2$ for any $e\ge 1$.

Next, we show that $\operatorname{ppt}(f)\le \alpha$.
Suppose that $\operatorname{ppt}(f)>\alpha$. The pair $(\Z_p\llbracket x,y\rrbracket,f^{\alpha})$ has klt singularities in the sense of \cite[Definition 2.6]{Ma-Schwede21} by \cite[Theorem 6.21]{Ma-Schwede21} since it is $+$-regular. Then the pair $((\Z_p[x,y])_{(p,x,y)},f^{\alpha})$ has klt singularities by \cite[Lemma 2.6]{Sato-Takagi25}. Since $(\Q_p[x,y])_{(x,y)}$ is a localization of $(\Z_p[x,y])_{(p,x,y)}$, the pair $((\Q_p[x,y])_{(x,y)},f^{\alpha})$ has klt singularities. Since $\operatorname{lct}((\Q_p[x,y])_{(x,y)},f)=\alpha$, this is a contradiction.
\end{proof}

As an application, we can give a stricter lower bound than the one given in \cite[Theorem C]{BJPRMSS25}.
\begin{eg}
	Let $S:=\Z_2\llbracket x, y \rrbracket$ and $f:=x^3+y^3+2^3$. Then $\ppt(f) \ge 2/3$. Indeed, since $f\equiv x^3+y^3 \pmod{2^2S}$, $\delta(f)\equiv \delta(x^3+y^3) \pmod{2S}$. Hence, we can perform the same calculation as Theorem \ref{diagonal hypersurface brt} and Corollary \ref{binomial ppt}.
\end{eg}

\subsection{Plus-pure thresholds and splitting-order sequences in mixed characteristic \texorpdfstring{$(0,2)$}{(0,2)}}
In this subsection, we study the relationship between our computation and splitting-order sequences, introduced by Yoshikawa \cite{Yoshikawa25b}, in mixed characteristic (0,2).
\begin{setting} \label{setting ppt p=2}
	Let $n$ be a positive integer, $S:=\Z_2\llbracket x_1,\dots, x_n \rrbracket$, $\n$ be the maximal ideal of $S$ and $f\in \n\setminus \{0\}$. We use $\overline{S}$ and $\overline{\n}$ to denote $S/pS$ and $\n\overline{S}$, respectively.
\end{setting}
First, we derive bounds on the BCM-regular and plus-pure thresholds using the p-th root formula.
\begin{defn}
	Let the notation be as in Setting \ref{setting ppt p=2}. For any $e\ge 1$, $s=(s_1,\dots,s_e)\in \{0,1\}^{e}$, we define 
\[
	\tau_{s}:=\prod_{i=1}^{e}\left((\phi^{-i}(f))^{1-s_i}(\phi^{-(i+1)}(\delta(f))\right)^{s_i})
\]
and
$2^s:=2^{\frac{s_1}{2}+\dots+\frac{s_e}{2^e}}$.
 For any $e\ge 1$, we define
	\[
		\mathcal{S}_e:=\{s\in \{0,1\}^{e}\mid \tau_s\notin \n S^+\}.
	\]
	We regard $\mathcal{S}_e$ as a totally ordered set by the lexicographic order.
\end{defn}
\begin{rem}
	The condition $\tau_s\notin \n S^+$ is equivalent to
	\[
		\prod_{i=1}^e\left( f^{(1-s_i) 2^{e+1-i}}(\delta(f))^{s_i2^{e-i}}\right)\notin \overline{\n}^{[2^{e+1}]}.
	\]
\end{rem}

\begin{prop} \label{prop set of sequences}
	With notation as in Setting \ref{setting ppt p=2}, let $1\le e_1\le e_2$.
The following statements hold:
	\begin{enumerate}
		\item If $(s_1,\dots, s_{e_2})\in \mathcal{S}_{e_2}$, then $(s_1,\dots,s_{e_1})\in \mathcal{S}_{e_1}$.
\item Suppose that $\mathcal{S}_{e_2}\neq \emptyset$. Let $(s_1,\dots,s_{e_2})=\min \mathcal{S}_{e_2}$. Then $(s_1,\dots,s_{e_1})=\min \mathcal{S}_{e_1}$.
	\end{enumerate}
\end{prop}
\begin{proof}
	The statement (1) is clear by definition. For (2), suppose that $s=(s_1,\dots,s_{e_2}):=\min \mathcal{S}_{e_2}$. By (1), we obtain $\mathcal{S}_{e_1}\neq \emptyset$. Let $t=(t_1,\dots,t_{e_1}):=\min \mathcal{S}_{e_1}$. Assume that $t<(s_1,\dots,s_{e_1})$. Then there exists $1\le j \le e_1$ such that $t_j=0<s_j=1$ and $s_i= t_i$ for $i<j$. Let 
\[
	g:=\prod_{i=1}^j\left( f^{(1-t_i) 2^{j-i}}(\delta(f))^{t_i2^{j-i-1}}\right).
\]
Then $g\notin \overline{\n}^{[2^{j}]}$.
Therefore, for any $m\ge 1$, we have
\[
	g^{1+2^{j}+\dots+2^{j(m-1)}} \notin \overline{\n}^{[2^{mj}]}.
\]
Indeed, there exists an $\overline{S}$-linear map $\Phi:F^{j}_*\overline{S}\to \overline{S}$ such that $F^{j}_*g \mapsto 1$. Then
\[
	\Phi\circ F^{j}_* \Phi \circ \dots \circ F^{(m-1)j}_*\Phi: F^{mj}_*\overline{S} \to \overline{S}
\]
maps $F^{mj}_*(g^{1+2^{j}+\dots+2^{j(m-1)}})$ to $1$, whence $g^{1+2^{j}+\dots+2^{j(m-1)}} \notin \overline{\n}^{[2^{mj}]}.$ Take $m\ge 1$ such that $mj\ge e_2$. For $1\le i \le j$ and $k \ge 1$, let $u_{kj+i}:=t_i$. Then $(u_1,\dots,u_{mj})\in \mathcal{S}_{mj}$ since $g^{1+2^{j}+\dots+2^{j(m-1)}} \notin \overline{\n}^{[2^{mj}]}.$ By (1), $u:=(u_1,\dots,u_{e_2})\in \mathcal{S}_{e_2}$. Since $u_i=t_i=s_i$ for $i<j$ and $u_j=t_j<s_j$, we obtain $u<s$, which is a contradiction.
\end{proof}

\begin{prop} \label{quasi-F-split case}
Let the notation be as in Setting \ref{setting ppt p=2}. Suppose that
\[
	f\delta(f)^{2^{e-1}-1}\notin \overline{\n}^{[2^{e}]}
\]
for some $e\ge 1$. Then $\ppt(f)=1$.
\end{prop}
\begin{proof}
	Let $e$ be the minimum positive integer $i$ such that $f\delta(f)^{2^{i-1}-1}\notin \overline{\n}^{[2^i]}$. Let $s_1=s_2=\dots =s_{e-1}=1$, $s_e=0$ and $s=(s_1,\dots,s_e)$. Then $s=\min \mathcal{S}_{e}$.
By the proof of Proposition \ref{prop set of sequences}, we have $\mathcal{S}_i\neq \emptyset$ for any $i\ge 1$. By Proposition \ref{prop set of sequences} (2), there exists $s=(s_1,s_2,\dots)\in \{0,1\}^{\N_{>0}}$ such that $(s_1,\dots,s_i)=\min \mathcal{S}_{i}$ for any $i\ge 1$.
By Lemma \ref{p-th root formula for ppt}, for any $i\ge 1$, there exists $\beta_i\in S^+$ such that
\begin{align*}
	f^{\frac{1}{2^i}} \equiv \phi^{-i}(f)+2^{\frac{1}{2^e}}\phi^{-(i+1)}(\delta(f))+2^{\frac{1}{2^{e}}+\frac{1}{2^{e+1}}}\phi^{-(i+1)}(f)\beta_i \pmod{2^{\frac{1}{2^{i-1}}}S^+}.
\end{align*}
Let $s^{(i)}:=(s_1,\dots,s_i)\in \mathcal{S}_i$ for any $i\ge 1$. 
It is enough to show the following claim.
\begin{cl}
	We have
	\begin{align*}
		f^{\frac{1}{2}+\dots+\frac{1}{2^{i}}} &\equiv 2^{s^{(i)}}\tau_{s^{(i)}} +2^{s^{(i-1)}}\cdot 2^{\frac{1}{2^i}+\frac{1}{2^{i+1}}}\tau_{s^{(i-1)}}\phi^{-(i+1)}(f)\beta_i \\
        & \pmod{(2^{s^{(i)}}\cdot2^{\frac{1}{2^i}},x_1,\dots.x_n) S^+}
	\end{align*}
	for any $i\ge 1$.
\end{cl}
\begin{clproof} 
We show the claim by induction on $i$.
Suppose that $i=1$. If $s_1=0$, then
\[
	f^{\frac{1}{2}} \equiv \phi^{-1}(f) \pmod{2^\frac{1}{2}S^+}.
\]
If $s_1=1$, then
\[
	f^{\frac{1}{2}} \equiv 2^{\frac{1}{2}}\phi^{-2}(\delta(f)) +2^{\frac{1}{2}+\frac{1}{4}}\phi^{-2}(f)\beta_1 \pmod{(2,x_1,\dots,x_n)S^+}.
\]
Hence, the case where $i=1$ follows. Suppose that $i\ge 2$. First, suppose that $s_i=0$.
Then we have
\begin{align*}
	f^{\frac{1}{2}+\dots+\frac{1}{2^i}} &= f^{\frac{1}{2}+\dots+\frac{1}{2^{i-1}}}f^{\frac{1}{2^{i}}} \\
 & \equiv 2^{s^{(i-1)}}\tau_{s^{(i-1)}}\phi^{-i}(f)\\
& \equiv 2^{s^{(i)}}\tau_{s^{(i)}} \pmod{(2^{s^{(i-1)}}\cdot 2^{\frac{1}{2^{i}}},x_1,\dots,x_n)S^+}.
\end{align*}
Next, suppose that $s_i=1$. Then we have
\begin{align*}
	f^{\frac{1}{2}+\dots+\frac{1}{2^i}} &= f^{\frac{1}{2}+\dots+\frac{1}{2^{i-1}}}f^{\frac{1}{2^{i}}} \\
	& \equiv 2^{s^{(i-1)}}\cdot 2^{\frac{1}{2^i}}\tau_{s^{(i-1)}}\phi^{-(i+1)}(\delta(f))+2^{s^{(i-1)}}\cdot 2^{\frac{1}{2^i}+\frac{1}{2^{i+1}}}\tau_{s^{(i-1)}} \phi^{-(i+1)}(f)\beta_i \\
&\equiv2^{s^{(i)}}\tau_{s^{(i)}}+2^{s^{(i-1)}}\cdot 2^{\frac{1}{2^i}+\frac{1}{2^{i+1}}}\tau_{s^{(i-1)}} \phi^{-(i+1)}(f)\beta_i\\
&\pmod{(2^{s^{(i)}}\cdot 2^{\frac{1}{2^{i}}},x_1,\dots,x_n)S^+}.
\end{align*}
Here note that $2^{s^{(i)}}\cdot 2^{\frac{1}{2^i}}=2^{s^{(i-1)}}\cdot 2^{\frac{1}{2^{i-1}}}$.
\end{clproof}
By the above cliam, we have
\begin{align*}
	f^{\frac{1}{2}+\dots+\frac{1}{2^{i}}} \notin \n S^+
\end{align*}
for any $i\ge 1$, which completes the proof. 
\end{proof}

\begin{prop} \label{prop perfectoid pure not quasi-F-split}
	Let the notation be as in Setting \ref{setting ppt p=2}. Suppose that
\[
	\delta(f)^{2^{i}-1}\notin \overline{\n}^{[2^{i+1}]}
\]
for any $1\le i \le e$. Then $\ppt(f)>1-1/2^e$.
\end{prop}
\begin{proof}
	By Proposition \ref{quasi-F-split case}, we may assume that	
\[
	f\delta(f)^{2^{i-1}-1}\in \overline{\n}^{[2^{i}]}
\]
for any $i\ge 1$. By an argument similar to Proposition \ref{quasi-F-split case}, we have
\begin{align*}
	f^{\frac{1}{2}+\cdots+\frac{1}{2^e}} &\equiv 2^{\frac{1}{2}+\cdots+\frac{1}{2^e}}\phi^{-2}(\delta(f))\cdots\phi^{-(e+1)}(\delta(f)) \\
& \quad +2^{\frac{1}{2}+\dots+\frac{1}{2^{e+1}}}\phi^{-2}(\delta(f))\cdots\phi^{-e}(\delta(f))\phi^{-(e+1)}(f)\beta_e \\
&\pmod{(2,x_1,\dots,x_n)S^+}.
\end{align*}
Since $\delta(f)^{2^{e}-1}\notin \overline{\n}^{[2^{e+1}]}$, we obtain
\[
	2^{\frac{1}{2}+\dots+\frac{1}{2^e}}\phi^{-2}(\delta(f))\cdots\phi^{-(e+1)}(\delta(f)) \notin (2^{\frac{1}{2}+\dots+\frac{1}{2^{e+1}}},x_1,\dots,x_n) S^+,
\]
which shows that $\ppt(f)> 1-1/2^e$.
\end{proof}

\begin{prop} \label{ppt splitting-order sequences}
	Let the notation be as in Setting \ref{setting ppt p=2}. Suppose that 
\[
	f\delta(f)^{2^{i-1}-1}\in \overline{\n}^{[2^{i}]}
\]
for any $i\ge 1$, and there exists $i\ge 1$ such that
\[
	\delta(f)^{2^{i}-1}\in \overline{\n}^{[2^{i+1}]}.
\]
Let $e$ be the minimum positive integer $i$ such that $\delta(f)^{2^{i}-1}\in \overline{\n}^{[2^{i+1}]}$ and $f\delta(f)^{2^{i}-2}\in \overline{\n}^{[2^{i+1}]}$. Then $(S,f^{1-1/2^e})$ is not $+$-regular.
\end{prop}
\begin{proof}
	By an argument similar to Proposition \ref{prop perfectoid pure not quasi-F-split}, we have
\begin{align*}
	f^{\frac{1}{2}+\cdots+\frac{1}{2^e}}
&\equiv 2^{\frac{1}{2}+\cdots+\frac{1}{2^e}}\phi^{-2}(\delta(f))\cdots\phi^{-(e+1)}(\delta(f))\\
& + 2^{\frac{1}{2}+\dots+\frac{1}{2^{e+1}}}\phi^{-2}(\delta(f))\cdots\phi^{-e}(\delta(f))\phi^{-(e+1)}(f)\beta_e \\
&\pmod{\n S^+}.
\end{align*}
Since $\delta(f)^{2^{e}-1}\in \overline{\n}^{[2^{e+1}]}$ and $f\delta(f)^{2^{e}-2}\in \overline{\n}^{[2^{e+1}]}$, we have
	$\phi^{-2}(\delta(f))\cdots\phi^{-(e+1)}(\delta(f))\in \n S^+$ and $\phi^{-2}(\delta(f))\cdots\phi^{-e}(\delta(f))\phi^{-(e+1)}(f)\in \n S^+$.
Hence,
\[
	f^{\frac{1}{2}+\cdots+\frac{1}{2^e}}\in \n S^+.
\]
Therefore, $(S,f^{1-1/2^e})$ is not $+$-regular.

\end{proof}

Next, we relate the above results to the notion of splitting-order sequences.

\begin{setting} \label{setting ppt p=2 b}
	With notation as in Setting \ref{setting ppt p=2}, let $u$ be a generator of $\Hom_{S}(F_*\overline{S},\overline{S})$ as an $F_*\overline{S}$-module.
\end{setting}

\begin{defn}[\cite{Yoshikawa25b}] \label{def of splitting-order sequences}
	Let the notation be as in Setting \ref{setting ppt p=2 b}. For integers $0\le l_1,\dots,l_{i-1}\le p-1$ and $0\le l_i\le p$, the ideal $I(l_1,\dots,l_i)$ of $\overline{S}$ is defined inductively by
\begin{itemize}
	\item $I(l_i):=f^{p-l_i}\overline{S}$,
	\item $I(l_1,\dots,l_i):=f^{p-l_1-1}u(F_*((\delta(f))^{l_1}I(l_2,\dots,l_i)))+f^{p-l_1}\overline{S}$.
\end{itemize}
The \textit{splitting-order sequence} $\mathbf{s}(f)=(s_0,s_1,\dots)$ of $f$ is defined inductively as follows:
\begin{itemize}
	\item $s_0:=0$,
	\item If $s_i=p$ for some $i$, then $s_j=p$ for all $j>i$,
	\item Suppose that $s_1,\dots,s_{i-1}\le p-1$. Then we define $s_i$ by
	\[
		s_i:=\max\{0\le s \le p \mid I(s_1,\dots,s_{i-1},s)\subseteq \m^{[p]}\}.
	\]
\end{itemize}
\end{defn}

\begin{lem} \label{lemma trace}
	Let $k$ be a perfect field and $(R,\m):=k\llbracket x_1,\dots,x_n\rrbracket$. Let $u$ be a generator of $\Hom_R(F_*R,R)$ as an $F_*R$-module,  $I$ be an ideal of $R$ and $i\ge 0$ be an integer. Then $u(F_*I)\subseteq \m^{[p^i]}$ if and only if $I\subseteq \m^{[p^{i+1}]}$.
\end{lem}
\begin{proof}
	If $I\subseteq \m^{[p^{i+1}]}$, then $F_*I\subseteq F_*\m^{[p^{i+1}]}=\m^{[p^i]}F_*R$. Hence, we have 
\[
u(F_*I)\subseteq \m^{[p^i]}\cdot u(F_*R)=\m^{[p^i]}.
\]
Suppose that $I\not \subseteq \m^{[p^{i+1}]}$. Take an element $f \in I\setminus \m^{[p^{i+1}]}$. There exists a monomial $g$ such that $gf=cx_1^{p^{i+1}-1}\cdots x_n^{p^{i+1}-1}+h$, where $c\in k \setminus \{0\}$ and $h\in \m^{[p^{i+1}]}$. Then we have
\begin{align*}
	u(F_*(gf)) &\equiv u(F_*(cx_1^{p^{i+1}-1}\cdots x_{n}^{p^{i+1}-1}))=x_1^{p^{i}-1}\cdots x_n^{p^{i}-1}u(F_*(cx_1^{p-1}\dots x_n^{p-1})) \\
&\pmod{\m^{[p^i]}}.
\end{align*}
Since $u(F_*(cx_1^{p-1}\dots x_n^{p-1}))$ is a unit of $R$, $u(F_*(gf))\notin \m^{[p^i]}$. Therefore, $u(F_*I)\not \subseteq \m^{[p^i]}$.
\end{proof}

\begin{prop} \label{prop splitting-order sequences}
	Let the notation be as in Setting \ref{setting ppt p=2 b} and $\mathbf{s}(f)=(s_0,s_1,\dots)$ be the splitting-order sequence of $f$.
\begin{enumerate}
	\item Suppose that $s_i \le 1$ for any $i\ge 1$. Then $(s_1,\dots,s_i)$ is the minimum element of $\mathcal{S}_i$ for any $i\ge 1$.
	\item Suppose that there exists an integer $i$ such that $s_i=2$. Let $e$ be the minimum integer such that $s_e=2$. Then $(\delta(f))^{2^{i-1}-1} \notin \overline{\n}^{[2^{i}]}$ for $i<e$ and $(\delta(f))^{2^{e-1}-1} \in \overline{\n}^{[2^{e}]}$.
\end{enumerate}
\end{prop}
\begin{proof}
For (1), firstly, suppose that $s_i=0$. Then $I(s_1,\dots,s_{i-1},1)\not\subseteq \overline{\n}^{[p]}$ by the definition of splitting-order sequences. Hence, by lemma \ref{lemma trace}, we have
\[
	\prod_{j=1}^{i}f^{(1-s_j)2^{i+1-j}}(\delta(f))^{s_j2^{i-j}}\not \in \overline{\n}^{[2^{i+1}]}.
\]
Hence, $(s_1,\dots,s_i)\in \mathcal{S}_i$.
Next, suppose that $s_i=1$. Since $s_{i+1}\le 1$, we have $I(s_1,\dots,s_i) \subseteq \overline{\n}^{[p]}$ and $I(s_1,\dots,s_{i},2)\not \subseteq \overline{\n}^{[p]}$ by the definition of splitting-order sequences. By lemma \ref{lemma trace}, we have
\[
	f^2\prod_{j=1}^{i-1}f^{(1-s_j)2^{i+1-j}}(\delta(f))^{s_j2^{i-j}} \in \overline{\n}^{[2^{i+1}]}
\]
and
\[
	\prod_{j=1}^{i}f^{(1-s_j)2^{i+1-j}}(\delta(f))^{s_j2^{i-j}} \not\in \overline{\n}^{[2^{i+1}]}.
\]
Hence, $(s_1,\dots,s_{i-1},0)\not \in \mathcal{S}_{i}$ and $(s_1,\dots,s_{i})\in \mathcal{S}_{i}$, which shows (1). For (2), by a similar argument and Proposition \ref{prop set of sequences},  it follows that $s_1=s_2=\dots=s_{e-1}=1$ and $s_e=2$. By the definition of splitting-order sequences, we have $I(s_1,\dots,s_{i-1},2)\not\subseteq \overline{\n}^{[p]}$ for $i<e$ and $I(s_1,\dots,s_{e-1},2)\subseteq \overline{\n}^{[p]}$. This implies that $(\delta(f))^{2^{i-1}-1}\not \in \overline{\n}^{[2^{i}]}$ for $i<e$ and $(\delta(f))^{2^{e-1}-1}\in \overline{\n}^{[2^e]}$.
\end{proof}

The following theorem is a refinement of \cite[Theorem B]{Yoshikawa25b} in mixed characteristic (0,2).
\begin{thm} \label{theorem splitting-order sequences and ppt}
	Let the notation be as in Setting \ref{setting ppt p=2 b} and $\mathbf{s}(f)=(s_0,s_1,\dots)$ be the splitting-order sequence of $f$.
\begin{enumerate}
	\item If $s_i \le 1$ for any $i\ge 1$, then $\operatorname{ppt}(f)= 1$.
	\item Suppose that there exists $i\ge 1$ such that $s_i=2$. Let $e\ge 1$ be the minimum integer $i$ such that $s_i=2$. Then
	\[
		1-\frac{1}{2^{e-2}}< \ppt(f)\le 1-\frac{1}{2^{e}}.
	\]
	Moreover, if $f(\delta(f))^{2^{e-1}-2} \in \overline{\n}^{[2^e]}$, then 
	\[
		1-\frac{1}{2^{e-2}}< \ppt(f) \le 1-\frac{1}{2^{e-1}}.
	\]
	\end{enumerate}
\end{thm}
\begin{proof}
	For (1), suppose that $s_i(f)\le 1$ for any $i\ge 1$. First, assume that $s_1=\dots=s_{e-1}=1$ and $s_e=0$. Then, by Proposition \ref{prop splitting-order sequences}, $(s_1,\dots,s_e)=\min \mathcal{S}_e$. Hence, $f\delta(f)^{2^{e-1}-1}\not \in \overline{\n}^{[2^e]}$. By Proposition \ref{quasi-F-split case}, $\ppt(f)=1$.
	Next, assume that $s_1=s_2=\dots=1$. By the proof of Proposition \ref{prop splitting-order sequences}, $(\delta(f))^{2^i-1}\not \in \overline{\n}^{[2^{i+1}]}$ for any $i\ge 1$. Hence, $\ppt(f)=1$ by Proposition \ref{prop perfectoid pure not quasi-F-split}.
	For (2), let $e$ be the minimum integer $i$ such that $s_i=2$. By Proposition \ref{prop splitting-order sequences}, $(\delta(f))^{2^{i-1}-1}\not \in \overline{\n}^{[2^i]}$ for $i < e$ and $(\delta(f))^{2^{e-1}-1}\in \overline{\n}^{[2^e]}$. Hence, by Proposition \ref{prop perfectoid pure not quasi-F-split}, we have
\[
	\ppt(f) > 1-\frac{1}{2^{e-2}}.
\]
Since $(\delta(f))^{2^{e-1}-1}\in \overline{\n}^{[2^e]}$, we see that $(\delta(f))^{2^{e}-2}\in \overline{\n}^{[2^{e+1}]}$. By Proposition \ref{ppt splitting-order sequences}, we have
\[
	\operatorname{ppt}(f)\le 1-\frac{1}{2^{e}}.
\]
Moreover, if $f(\delta(f))^{2^{e-1}-2}\in \overline{\n}^{[2^e]}$, then, by Proposition \ref{ppt splitting-order sequences},
\[
	\operatorname{ppt}(f)\le 1-\frac{1}{2^{e-1}}.
\]
\end{proof}

\begin{eg}
	In the above theorem, the condition $f(\delta(f))^{2^{e-1}-1}\in \overline{\n}^{[2^e]}$ is necessary. Indeed, let $S:=\Z_2\llbracket x, y \rrbracket$, $\n=(2,x,y)S$, $f=x^2+y^3$ and $g=x^3+y^3$. Then $\mathbf{s}(f)=\mathbf{s}(g)=(0,1,1,2,\dots)$. On the other hand,
$f(\delta(f))^2=(x^2+y^3)x^4y^6 \not \in \overline{\n}^{[2^3]}$, but 
$g(\delta(g))^2=(x^3+y^3)x^6y^6 \in \overline{\n}^{[2^3]}$. By Corollary \ref{binomial ppt}, we see that $\operatorname{ppt}(f)=5/6$ and $\operatorname{ppt}(g)=2/3$. Here
\[
	1-\frac{1}{2}<\operatorname{ppt}(g)\le 1-\frac{1}{4},
\]
but $\operatorname{ppt}(f)> 3/4$.
\end{eg}

\subsection{BCM-regularity of diagonal hypersurfaces in mixed characteristic (0,2)}
In this subsection, as an application of the above theorems, we determine which diagonal hypersurfaces are BCM-regular in mixed characteristic (0,2).
\begin{setting} \label{setting p=2}
	Let $n\ge 2$, $2\le d_0\le \dots \le d_n$ and $\mathbf{d}=(d_0,\dots,d_n)$. Let $f_{\mathbf{d}}:=x_0^{d_0}+\dots +x_n^{d_n}$ and $R_\mathbf{d}:=\Z_{2}\llbracket x_0,\dots, x_n\rrbracket/(f_{\mathbf{d}})$. 
\end{setting}

\begin{lem} \label{lemma classification}
	Let the notation be as in Setting \ref{setting p=2}.
\begin{enumerate}
\item The ring $R_{\mathbf{d}}$ is BCM-regular if one of the following conditions holds.
\begin{enumerate}[label=(\roman*)]
	\item $\mathbf{d}=(2,2,m)$ and $m\ge 2$.
	\item $\mathbf{d}=(2,3,5)$.
	\item $\mathbf{d}=(2,3,6,m)$ and $m\ge 6$.
	\item $\mathbf{d}=(2,3,7,31)$.
	\item $\mathbf{d}=(3,3,3,m)$ and $m\ge 3$.
	\item $\mathbf{d}=(3,3,6,6,m)$ and $m\ge 6$.
	\item $\mathbf{d}=(3,3,5,7)$.
	\item $\mathbf{d}=(3,3,6,7,31)$.
	\item $\mathbf{d}=(3,3,7,7,15)$.
	\item $\mathbf{d}=(3,3,7,7,31,31)$.
\end{enumerate}
	\item The ring $R_{\mathbf{d}}$ is not perfectoid pure if one of the following conditions holds.
	\begin{enumerate}[label=(\roman*)]
		\item $(d_0,d_1)=(2,3)$ and $d_2\ge 8$.
		\item $(d_0,d_1,d_2)=(2,3,7)$ and $d_3\ge 32$.
		\item $(d_0,d_1,d_2)=(3,3,4)$ and $d_3\ge 8$.
		\item $\mathbf{d}=(3,3,7,7,16)$.
		\item $(d_0,d_1,d_2,d_3,d_4)=(3,3,7,7,16)$ and $d_5\ge 32$.
		\item $(d_0,d_1,d_2,d_3)=(3,3,6,7)$ and $d_4\ge 32$.
	\end{enumerate}
\end{enumerate}
\end{lem}
\begin{proof}
	(1)(i) By Corollary \ref{binomial ppt}, $\ppt(f_{(2,2)})=1$. Hence, $R_{(2,2,m)}$ is BCM-regular for any $m\ge 1$ by Proposition \ref{BCM-regularity and brt}.

(ii) By Corollary \ref{binomial ppt}, $\ppt(f_{(2,3)})=5/6$. Hence, by Proposition \ref{BCM-regularity and brt}, $R_{(2,3,5)}$ is BCM-regular.

(iii) Since we can take $\alpha_1=1/2$, $\alpha_2=1/3$ and $\alpha_{3}=1/6$ in Theorem \ref{diagonal hypersurface brt}, it follows that $\ppt(f_{(2,3,6)})=1$. Hence, $R_{(2,3,6,m)}$ is BCM-regular for any $m\ge 1$ by Proposition \ref{BCM-regularity and brt}.

(iv) Similarly, 
	\[
		\ppt(f_{(3,7,31)})=\frac{1}{3}+\frac{1}{7}+\frac{1}{31}=\frac{331}{651}>\frac{1}{2}.
	\]
	Hence, $R_{(2,3,7,31)}$ is BCM-regular.

(v) Similarly, $\ppt(f_{(3,3,3)})=1$. Hence, $R_{(3,3,3,m)}$ is BCM-regular for any $m\ge 3$.

(vi) Similarly, $\ppt(f_{(3,3,6,6)})=1$. Hence, $R_{(3,3,6,6,m)}$ is BCM-regular for any $m\ge 6$.

(vii) Similarly, 
\[
	\ppt(f_{(3,3,5)})=\frac{1}{3}+\frac{1}{3}+\frac{1}{5}=\frac{13}{15}>\frac{6}{7}.
\]
Hence, $R_{(3,3,5,7)}$ is BCM-regular.

(viii) For $f_{(3,3,6,7)}$, we can take $\alpha_1=1/3$, $\alpha_2=1/3$, $\alpha_4=1/7$ and
\[
	\alpha_3=\frac{1}{2^3}+\sum_{i=1}^{\infty} \left(\frac{1}{2^{6i-1}}+\frac{1}{2^{6i+3}}\right)=\frac{10}{63}
\]
in Theorem \ref{diagonal hypersurface brt}.
Hence, 
\[
	\ppt(f_{(3,3,6,7)})\ge \frac{1}{3}+\frac{1}{3}+\frac{10}{63}+\frac{1}{7}=\frac{61}{63}>\frac{30}{31}.
\]
Therefore, $R_{(3,3,6,7,31)}$ is BCM-regular.

(ix) For $f_{(3,3,7,7)}$, we can take $\alpha_1=1/3$, $\alpha_2=1/3$, $\alpha_3=1/7$ and
\[
	\alpha_4=\sum_{i=1}^\infty \frac{1}{2^{6i-3}}= \frac{8}{63}
\]
in Theorem \ref{diagonal hypersurface brt}. Hence,
\[
	\ppt(f_{(3,3,7,7)}) \ge \frac{1}{3}+\frac{1}{3}+\frac{1}{7}+\frac{8}{63}=\frac{59}{63}>\frac{14}{15}.
\]
Therefore, $R_{(3,3,7,7,15)}$ is BCM-regular.

(x) For $f_{(3,3,7,7,31)}$, we can take $\alpha_1=1/3$, $\alpha_2=1/3$, $\alpha_3=1/7$,
\[
	\alpha_4=\sum_{i=1}^\infty \frac{1}{2^{6i-3}}=\frac{8}{63}
\]
and
\[
	\alpha_5=\sum_{i=0}^\infty \frac{1}{2^{30i+5}}=\frac{2^{25}}{2^{30}-1}.
\]
Hence, 
\[
	\ppt(f_{(3,3,7,7,31)}) \ge \frac{1}{3}+\frac{1}{3}+\frac{1}{7}+\frac{8}{63}+\frac{2^{25}}{2^{30}-1}>\frac{30}{31}.
\]
Therefore, $R_{(3,3,7,7,31,31)}$ is BCM-regular.

(2)(i) Since $\delta(f_{\mathbf{d}})\equiv x_0^2x_1^3 \pmod{\overline{\n}^{[8]}}$, we have $\delta(f_{\mathbf{d}})\not \in \overline{\n}^{[4]}$, $f_{\mathbf{d}}\delta(f_{\mathbf{d}})\in \overline{\n}^{[4]}$ and $(\delta(f_{\mathbf{d}}))^3\in \overline{\n}^{[8]}$. Therefore, $\mathbf{s}(f_{\mathbf{d}})=(0,1,1,2,\dots)$. Hence, $R_{\mathbf{d}}$ is not perfectoid pure by Theorem \ref{theorem splitting-order sequences and ppt} (2).

(ii) Since $\delta(f_{\mathbf{d}}) \equiv x_0^2x_1^3+x_0^2x_2^7+x_1^3x_2^7 \pmod{\overline{\n}^{[32]}}$, we have $(\delta(f_{\mathbf{d}}))^{2^i-1}\not \in \overline{\n}^{[2^{i+1}]}$, $f_{\mathbf{d}}(\delta(f_{\mathbf{d}}))^{2^i-1} \in \overline{\n}^{[2^{i+1}]}$ for $i=1, 2, 3$ and $(\delta(f_{\mathbf{d}}))^{15} \in \overline{\n}^{[32]}$. Hence, $\mathbf{s}(f_{\mathbf{d}})=(0,1,1,1,1,2,\dots)$. Hence, $R_{\mathbf{d}}$ is not perfectoid pure by Theorem \ref{theorem splitting-order sequences and ppt} (2).

(iii) Similarly, $\mathbf{s}(f_{\mathbf{d}})=(0,1,1,2,\dots)$. Hence, $R_\mathbf{d}$ is not perfectoid pure.

(iv)(v)(vi) Similarly, $\mathbf{s}(f_{\mathbf{d}})=(0,1,1,1,1,2,\dots)$. Hence, $R_{\mathbf{d}}$ is not perfectoid pure.
\end{proof}

\begin{notation}
	Let $i, m_0, \dots, m_i\in \N$.
	We use the notation $(m_0,\dots,m_i,*)$ to denote an arbitrary element of the set
\[
\bigcup_{n=i}^{\infty}\{ (d_0,\dots,d_n)\in \N^{n+1}\mid \text{$(d_0,\dots,d_i)=(m_0,\dots,m_i)$ and $d_0\le \dots \le d_n$}\}.
\]
\end{notation}

\begin{thm} \label{theorem classification}
	Let the notation be as in Setting \ref{setting p=2}. Then $R_{\mathbf{d}}$ is $+$-regular if and only if it is BCM-regular, and this holds if and only if $\mathbf{d}$ is in the following list.
	\begin{enumerate}
		\item $(2,2,m,*)$ for some $m\ge 2$.
		\item $(2,3,m,*)$. where $3\le m \le 5$.
		\item $(2,3,6,m,*)$, where $m\ge 6$.
		\item $(2,3,7,m,*)$, where $7 \le m \le 31$.
		\item $(3,3,3,m,*)$, where $m\ge 3$.
		\item $(3,3,m_1,m_2,*)$, where $m_1\le 5$ and $m_2\le 7$
		\item $(3,3,6,6,m,*)$, where $m\ge 6$.
		\item $(3,3,6,7,m,*)$, where $7 \le m \le 31$.
		\item $(3,3,7,7,m,*)$, where $7\le m \le 15$.
		\item $(3,3,7,7,m_1,m_2,*)$, where $7\le m_1\le m_2\le 31$.
	\end{enumerate}
\end{thm}
\begin{proof}
	In this proof, we freely use Proposition \ref{proposition index inequality} and \cite[Theorem 6.27]{Ma-Schwede21}.
	If $d_1\ge 4$, then $R$ is not $+$-regular by the proof of Proposition \ref{Fermat type >=p^2}. Hence, $(d_0,d_1)=(2,2), (2,3), (3,3)$ if $R$ is $+$-regular. 

 First, suppose that $(d_0,d_1)=(2,2)$. Then $R$ is BCM-regular by Lemma \ref{lemma classification} (1)(i). 

 Next, suppose that $(d_0,d_1)=(2,3)$. If $d_2\le 5$, then $R$ is BCM-regular by Lemma \ref{lemma classification} (1)(ii). If $n=2$ and $d_2\ge 6$, then $R$ is not $+$-regular by Corollary \ref{+-regular and canonical diagonal hypersurface}. Suppose that $n\ge 3$. If $d_2\ge 8$, then $R$ is not perfectoid pure by Lemma \ref{lemma classification} (2)(i). If $d_2=6$, then $R$ is BCM-regular by Lemma \ref{lemma classification} (iii). Suppose that $d_2=7$. In this case, $R$ is not perfectoid pure if $d_3\ge 32$ by Lemma \ref{lemma classification} (2)(ii), and $R$ is BCM-regular if $d_3\le 31$ by Lemma \ref{lemma classification} (1)(iv).

 Lastly, suppose that $(d_0,d_1)=(3,3)$. If $n=2$, then $R$ is not $+$-regular by Corollary \ref{+-regular and canonical diagonal hypersurface}. Suppose that $n\ge 3$. If $d_2=3$, then $R$ is BCM-regular by Lemma \ref{lemma classification} (1)(v). Suppose that $d_2\ge 4$. By Lemma \ref{lemma classification} (2)(iii), $R$ is not $+$-regular if $d_3\ge 8$. If $d_2\le 5, d_3\le 7$, then $R$ is BCM-regular by Lemma \ref{lemma classification} (1)(vii). If $n=3$ and $d_2\ge 6$, then $R$ is not $+$-regular by Corollary \ref{+-regular and canonical diagonal hypersurface}. Hence, we may assume that $n\ge 4$. The only remaining cases are $(d_2,d_3)=(6,6), (6,7)$ and $(7,7)$.

If $(d_2,d_3)=(6,6)$, then $R$ is BCM-regular by Lemma \ref{lemma classification} (1)(vi). If $(d_2, d_3)=(6,7)$, then $R$ is BCM-regular if and only if $d_4\le 31$ by Lemma \ref{lemma classification} (1)(viii) and (2)(vi). If $(d_2,d_3)=(7,7)$ and $n=4$, then $R$ is BCM-regular if and only if $d_4\le 15$ by Lemma \ref{lemma classification} (1)(ix) and (2)(iv). If $(d_2,d_3)=(7,7)$, $d_4\ge 16$ and $n\ge 5$, then $R$ is BCM-regular if and only if $d_4\le d_5\le 31$ by Lemma \ref{lemma classification} (1)(x) and (2)(v). In the above cases, $R$ is BCM-regular if and only if $R$ is $+$-regular. Summarizing this argument, we obtain the desired list.
\end{proof}
\begin{cor} \label{cor classification}
	Let $R=\Z_{2}\llbracket x_0,\dots, x_n\rrbracket/(x_0^{d_0}+\dots +x_n^{d_n})$. Suppose that $n\ge 1$ and $2\le d_0\le \dots \le d_n$. Then $R$ is perfectoid pure, but not $+$-regular if and only if $\mathbf{d}$ is equal to one of the following.
\begin{enumerate}
	\item (2,2).
	\item (2,3,6).
	\item (3,3,3).
	\item (3,3,6,6).
\end{enumerate}
\end{cor}
\begin{proof}
	Let $S:=\Z_2\llbracket x_0,\dots,x_n\rrbracket$ and $f:=x_0^{d_0}+\dots+x_n^{d_n}$. Then $R$ is perfectoid pure if and only if $\operatorname{ppt}(S,f)=1$ by \cite[Proposition 2.8]{BJPRMSS25}. This holds if and only if $S\xrightarrow{f^{(m-1)/m}} S^+$ splits for any $m\ge 1$. Moreover, this is equivalent to the $+$-regularity of the ring $S\llbracket x_{n+1}\rrbracket/(f+x_{n+1}^m)$ by Remark \ref{remark +-regular and ppt}. Hence, the result follows from Theorem \ref{theorem classification}.
\end{proof}

\bibliography{mixed.bib}
\bibliographystyle{alpha}

\bigskip

\end{document}